\documentclass{article}
\usepackage{amsmath}
\usepackage{amssymb}
\usepackage{amsfonts}
\usepackage{eucal}
\usepackage{url}
\usepackage{bm}
\usepackage[dvipdfmx]{hyperref}
\usepackage{mathrsfs}
\usepackage[dvipdfmx]{graphicx}
\usepackage{amscd}
\usepackage{amsthm}
\usepackage[all]{xy}
\usepackage{color}

\newcommand{\Ad}{\mathop{\mathrm{Ad}}}
\newcommand{\Aut}{\mathop{\mathrm{Aut}}\nolimits}
\newcommand{\CAlg}{\mathcal{CA}\mathrm{lg}}

\newcommand{\cmod}{\textrm{-}\mathrm{mod}}
\newcommand{\comod}{\textrm{-}{\rm comod}}
\newcommand{\End}{\mathop{\mathrm{End}}\nolimits}
\newcommand{\Ext}{\mathop{\mathrm{Ext}}\nolimits}
\newcommand{\Frac}{\mathop{\mathrm{Frac}}\nolimits}
\newcommand{\Grp}{\mathcal{G}\mathrm{rp}}
\newcommand{\Hom}{\mathop{\mathrm{Hom}}\nolimits}
\newcommand{\Ker}{\mathop{\mathrm{Ker}}\nolimits}

\newcommand{\ind}{\mathop{\mathrm{ind}}\nolimits}
\newcommand{\Ind}{\mathop{\mathrm{Ind}}\nolimits}

\newcommand{\pro}{\mathop{\mathrm{pro}}\nolimits}
\newcommand{\Res}{\mathop{\mathrm{Res}}\nolimits}

\newcommand{\Spec}{\mathop{\mathrm{Spec}}}
\newcommand{\Coh}{\mathop{\mathrm{Coh}}}

\newcommand{\GL}{\mathop{\mathrm{GL}}\nolimits}

\newcommand{\SL}{\mathop{\mathrm{SL}}\nolimits}
\newcommand{\SO}{\mathop{\mathrm{SO}}\nolimits}

\newcommand{\bC}{\mathbb{C}}

\newcommand{\bQ}{\mathbb{Q}}
\newcommand{\bR}{\mathbb{R}}
\newcommand{\bZ}{\mathbb{Z}}
\newcommand{\cA}{\mathcal{A}}
\newcommand{\cB}{\mathcal{B}}
\newcommand{\cC}{\mathcal{C}}

\newcommand{\cF}{\mathcal{F}}

\newcommand{\cI}{\mathcal{I}}
\newcommand{\cJ}{\mathcal{J}}

\newcommand{\cO}{\mathcal{O}}

\newcommand{\fb}{\mathfrak{b}}
\newcommand{\fg}{\mathfrak{g}}
\newcommand{\fgl}{\mathfrak{gl}}
\newcommand{\fh}{\mathfrak{h}}
\newcommand{\fk}{\mathfrak{k}}
\newcommand{\fl}{\mathfrak{l}}

\newcommand{\fq}{\mathfrak{q}}

\newcommand{\fu}{\mathfrak{u}}

\theoremstyle{plain}
\newtheorem{thm}{Theorem}[subsection]
\newtheorem{cor}[thm]{Corollary}
\newtheorem{lem}[thm]{Lemma}
\newtheorem{prop}[thm]{Proposition}
\theoremstyle{definition}

\newtheorem{cond}[thm]{Condition}
\newtheorem{cons}[thm]{Construction}
\newtheorem{conv}[thm]{Convention}
\newtheorem{defn}[thm]{Definition}
\newtheorem{ex}[thm]{Example}
\newtheorem{note}[thm]{Notation}

\newtheorem{rem}[thm]{Remark}
\newtheorem{var}[thm]{Variant}
\newtheorem{LemA}{Lemma}
\newtheorem{ThmB}{Theorem}
\newtheorem{ThmC}{Theorem}
\newtheorem{ThmD}{Theorem}
\newtheorem{ThmE}{Theorem}
\newtheorem{PropF}{Proposition}
\newtheorem{VarG}{Variant}
\newtheorem{ThmH}{Theorem}

\begin{document}
\title{Flat Base Change Formulas for $(\fg,K)$-modules over Noetherian rings}
\author{Takuma Hayashi\thanks{Graduate School of Mathematical Sciences, The University of Tokyo, 3-8-1 Komaba Meguro-ku Tokyo 153-8914, Japan, htakuma@ms.u-tokyo.ac.jp}}
\date{}
\maketitle
\begin{abstract}
We discuss the flat base change formulas of the functor $I$ and its derived functor. In particular, a flat base change theorem for $A_\fq(\lambda)$ is obtained. 
\end{abstract}
\section{Introduction}
\subsection{Background and aims}
The theory of $(\fg,K)$-modules is an algebraic approach to representation theory of real reductive Lie groups. Recently, integral and rational structures of real reductive groups and their representations have been studied by M.\ Harris, G.\ Harder, F.\ Januszewski, and the author (\cite{MR3053412}, \cite{10.1093/imrn/rny043}, \cite{1407.0574}, \cite{1604.04253}, \cite{MR3770183}, \cite{MR4007195}, and \cite{1606.04320} for example). J.\ Bernstein et al.\ also introduced contraction families as pairs over the polynomial ring $\bC\left[z\right]$ in \cite{10.1093/imrn/rny146} and \cite{10.1093/imrn/rny147}. These are all regarded as a part of the theory of $(\fg,K)$-modules over commutative rings.

The functor $I^{\fg,K}_{\fq,M}$ and its derived functor are an important construction of $(\fg,K)$-modules over the complex number field $\bC$. In particular, they include an algebraic analog of real parabolic induction, and produce the so-called $A_\fq(\lambda)$-modules which are discrete series representations of real semisimple Lie groups in special cases. If we are given a map $(\fq,M)\to(\fg,K)$ of pairs, the functor $I^{\fg,K}_{\fq,M}$ is right adjoint to the forgetful functor $\cF^{\fq,M}_{\fg,K}$ from the category of $(\fg,K)$-modules to that of $(\fq,M)$-modules. Its derived functor can be computed by the standard resolution which is obtained from the Koszul complex (see \cite{MR1330919} for details).

Januszewski constructed the functor $I^{\fg,K}_{\fq,M}$ and its derived functor in a similar way to the complex case when the base ring is a field of characteristic 0 and the groups $K,M$ are reductive (\cite{1604.04253}, \cite{MR3770183}). In a view from homological algebra, this cannot be generalized in a straightforward way when the base ring is no longer a field. For integral structures, Harder suggests to replace $\bC$ by the ring $\bZ$ of integers in the standard resolution for a definition of the $(\fg,K)$-cohomology. In \cite{MR4007195} and \cite{1606.04320}, the author constructed the functor $I^{\fg,K}_{\fq,M}$ and its derived functor over an arbitrary commutative ring. The arguments of \cite{MR4007195} heavily rely on generalities on categories, especially, closed symmetric monoidal categories. Though we know the existence of the functor, we did not understand what they actually produced.

The study of the functor $I^{\fg,K}_{\fq,M}$ consist of three steps:
\begin{enumerate}
\item[(A)] Construct $I^{\fg,K}_{\fq,M}$ and its derived functor, or prove their existence.
\item[(B)] Find pairs and suitable $(\fq,M)$-modules which are meaningful to representation theory of real reductive groups.
\item[(C)] Study the resulting $(\fg,K)$-modules from the functor $I^{\fg,K}_{\fq,M}$.
\end{enumerate}
Generalities on Part (A) were established by \cite{MR4007195} and \cite{1606.04320} as mentioned above. Part (B) is well-studied in principle when the base is $\bC$ (see \cite{MR1330919}). Usually, $(\fg,K)$ may be the Harish-Chandra pair associated to a real reductive group, and $(\fq,M)$ may be real or $\theta$-stable parabolic subpairs. However, if we work over $\bZ$, we will have many choices of $\bZ$-forms of such pairs over $\bZ$. This problem will be related to explicit descriptions in Part (C). The main purpose of this paper is to work on Part (C) in an abstract way.

\begin{note}\label{1.1.1}
For $(\fg,K)$-modules $V$ and $V'$, write $\Hom_{\fg,K}(V,V')$ for the $k$-module of $(\fg,K)$-homomorphisms from $V$ to $V'$. We will use similar notations for modules over other algebraic objects like $K$-modules.
\end{note}
\begin{conv}\label{1.1.2}
The tensor products without decoration are understood to be over the base ring $k$.
\end{conv}
In \cite{MR3770183}, Januszewski discussed the behavior of $\Ext^\bullet_{\fg,K}$ the functor $I^{\fg,K}_{\fq,M}$ along extensions $k'/k$ of fields, and proved the base change formulas
\[\Hom_{\fg,K}(X,-)\otimes k'\cong\Hom_{\fg\otimes k',K\otimes k'}(X\otimes k',-\otimes k')\]
\[\Ext^\bullet_{\fg,K}(-,-)\otimes k'\cong \Ext_{\fg\otimes k',K\otimes k'}(-\otimes k',-\otimes k')\]
\[\bR I^{\fg,K}_{\fq,M}(V)\otimes k'\simeq\bR I^{\fg\otimes k',K\otimes k'}_{\fq\otimes k',M\otimes k'}(V\otimes k')\]
\[H^\bullet(\fg,K,-)\otimes k'\cong H^\bullet(\fg\otimes k',K\otimes k',-\otimes k')\]
under suitable finiteness conditions (see \cite{MR3770183} for details). This can be regarded as Part (C). He also considered rational forms of cohomological inductions (Part (B) and Part (C), see \cite{MR3770183} 7.1).

In \cite{MR2015057}, base change formulas of representations of affine group schemes $K$ over commutative rings $k$ are discussed.
\begin{note}\label{1.1.3}
If $M\to K$ is a homomorphism between flat affine group schemes over $k$, let us denote the right adjoint functor to the forgetful functor from the category of $K$-modules to that of $M$-modules by $\Ind^K_M$.
\end{note}
Let $K$ be an affine group scheme, and $V$ be a $K$-module which is finitely generated and projective as a $K$-module. According to \cite{MR2015057} I.2.10, we have $\Hom_K(V,-)\otimes k'\cong\Hom_{K\otimes k'}(V\otimes k',-\otimes k')$. Moreover, if $k'$ is finitely generated and projective as a $k$-module then the isomorphism above holds for small colimits of such $V$. For a homomorphism $M\to K$ between flat affine group schemes over $k$, the group cohomology $H^\bullet(K,-)$ and the cohomology functor $R^n\Ind^K_M(-)$ respects flat base changes (loc.\ cit.\ Proposition I.4.13).

Our goal is to establish these isomorphisms along flat homomorphisms from Noetherian rings. Supplementarily, we also work again on Part (A) to relax the definition of pairs in \cite{MR4007195}. Fix $k$ as a commutative ground ring.
\begin{cond}\label{1.1.4}
\begin{enumerate}
\renewcommand{\labelenumi}{(\arabic{enumi})}
\item A $k$-module $V$ is said to satisfy Condition \ref{1.1.4} (1) if for any flat commutative $k$-algebra $R$, the canonical homomorphism
\[\Hom_k(V,k)\otimes R\to\Hom_k(V,R)\]
is an isomorphism.
\item A $k$-module $V$ is said to satisfy Condition \ref{1.1.4} (2) if for any $k$-module $W$ and any flat commutative $k$-algebra $R$ , the canonical homomorphism
\[\Hom_k(V,W)\otimes R\to\Hom_k(V,W\otimes R)\]
is an isomorphism.
\end{enumerate}
\end{cond}
\begin{ex}\label{1.1.5}
Finitely presented $k$-modules $V$ satisfy Condition \ref{1.1.4} (2).
\end{ex}
\begin{cond}\label{1.1.6}
Let $K$ be a flat affine group scheme over $k$. Write $I_e$ for the kernel of the counit of the coordinate ring of $K$. Then $K$ is said to satisfy Condition \ref{1.1.6} if the $k$-modules $I_e/I_e^2$ and its dual $\fk=\Hom_k(I_e/I_e^2,k)$ enjoy Condition \ref{1.1.4} (1) and (2) respectively.
\end{cond}
\begin{ex}\label{1.1.7}
If $k$ is Noetherian, and $I_e/I_e^2$ is finitely generated then $\fk$ is also finitely generated. In particular, both $I_e/I_e^2$ and $\fk$ satisfy Condition \ref{1.1.4} (2).
\end{ex}
\begin{note}\label{1.1.8}
For a flat affine group scheme satisfying Condition \ref{1.1.6}, its Lie algebra will be denoted by the corresponding small German letter.
\end{note}
A pair consists of a flat affine group scheme $K$ satisfying Condition \ref{1.1.6} and a $k$-algebra $\cA$ with a $K$-action $\phi$, equipped with a $K$-equivariant Lie algebra homomorphism $\psi:\fk\to\cA$. Moreover, a pair is demanded to satisfy the equality $d\phi(\xi)=\left[\psi(\xi),-\right]$ for any $\xi\in\fk$, where
$d\phi$ is the differential representation of $\phi$. The point of modification from \cite{MR4007195} is on the condition of $I_e/I_e^2$. In loc.\ cit., we required that $I_e/I_e^2$ is finitely generated and projective (\cite{MR4007195} Condition 2.2.2). For a pair $(\cA,K)$, an $(\cA,K)$-module is a $K$-module, equipped with a $K$-equivariant $\cA$-module structure such that the two induced actions of $\fk$ coincide (\cite{MR4007195}). We consider a version to replace algebras $\cA$ by Lie algebras $\fg$. Remark that in this paper, we do not discuss differential graded modules like loc.\ cit. Then the same arguments as \cite{MR4007195} Section 2.3 and \cite{1606.04320} Theorem 2.1.1 still work.
\begin{lem}\label{1.1.9}
Let $(\cA,K)\to(\cB,L)$ be a map of pairs in the above sense. Then we have a forgetful functor $\cF^{\cA,K}_{\cB,L}$ from the category of $(\cB,L)$-modules to that of $(\cA,K)$-modules, which admits a right adjoint functor $I^{\cB,L}_{\cA,K}$. 
\end{lem}

In \cite{1712.07336}, we focus on Part (B) and Part (C). In that paper, we consider the cases where the group $K$ is a torus. Then the theory of Hecke algebras and the Koszul resolutions work well. For instance, we study integral models of the pairs associated to the finite covering groups of $\rm{PU}(1,1)$ and their principal series representations and discrete series representations. We see that whether the induced modules from $I^{\fg,K}_{\fq,M}$ vanish or not really depends on the choice of $\bZ$-forms of pairs over $\bC$.
\subsection{Main Results}
In this paper, the pairs $(\cA,K)$ we mainly consider arise from their versions $(\fg,K)$ for Lie algebras through $\cA=U(\fg)$ the enveloping algebra of $\fg$. Therefore we write $(\fg,K)$-modules for $(U(\fg),K)$-modules.
\begin{note}\label{1.2.1}
For a pair $(\fg,K)$ over $k$, denote the category of $(\fg,K)$-modules by $(\fg,K)\cmod$.
\end{note}
We next introduce the functors of flat base changes. Let $k\to k'$ be a flat homomorphism of commutative rings, and $(\fg,K)$ be a pair over $k$.
\begin{LemA}[Proposition \ref{3.1.1}]\label{A}
\begin{enumerate}
\renewcommand{\labelenumi}{(\arabic{enumi})}
\item The Lie algebra $\fg\otimes k'$ and the affine group scheme $K\otimes k'$ over $k'$ naturally form a pair $(\fg\otimes k',K\otimes k')$ over $k'$.
\item The extension and the restriction of scalars of modules extend to an adjunction
\[-\otimes_k k':(\fg,K)\cmod\leftrightarrows(\fg\otimes k',K\otimes k')\cmod:\Res^k_{k'}.\]
\end{enumerate}
\end{LemA}
Assume $k$ to be Noetherian. We compare a relation of $\Hom$ modules and flat base changes.
\begin{ThmB}[Flat base change theorem, Theorem \ref{3.1.6}]\label{B}
Suppose that $\fg$ is finitely generated as a $k$-module. Then for any finitely generated $(\fg,K)$-module $X$, we have a natural isomorphism
\[\Hom_{\fg,K}(X,-)\otimes k'\cong\Hom_{\fg\otimes k',K\otimes k'}(X\otimes k',-\otimes
 k').\]
\end{ThmB}
\begin{ThmC}[Theorem \ref{3.1.7}]\label{C}
Let $k\to k'$ be a flat ring homomorphism, and $(\fq,M)\to(\fg,K)$ be a map of pairs. Suppose that the following conditions are satisfied:
\begin{enumerate}
\renewcommand{\labelenumi}{(\roman{enumi})}
\item $\fk\oplus\fq\to\fg$ is surjective.
\item $\fq$ and $\fg$ are finitely generated as $k$-modules.
\end{enumerate}
Then we have an isomorphism
\[(I^{\fg,K}_{\fq,M}V)\otimes_k k'\cong I^{\fg\otimes_k k',K\otimes_k k'}_{\fq\otimes_k k',M\otimes_k k'}(V\otimes_k k').\]
\end{ThmC}
\begin{ex}\label{1.2.2}
Let $k$ be the ring $\bZ$ of integers, and $k'$ be the field $\bQ$ of rational numbers. Then Theorem \ref{C} asserts that $I^{\fg,K}_{\fq,M}(V)$ is a $\bZ$-form (with torsions) of $I^{\fg\otimes \bQ,K\otimes \bQ}_{\fq\otimes \bQ,M\otimes \bQ}(V\otimes \bQ)$.
\end{ex}
The condition (i) of Theorem \ref{C} is satisfied in the following cases:
\begin{ex}[The Zuckerman functor]\label{1.2.3}
The Lie algebra $\fq$ is equal to $\fg$, and the map $\fq\to\fg$ is the identity. In this case, $\Gamma=I^{\fg,K}_{\fg,M}$ is called the Zuckerman functor.
\end{ex}
\begin{ex}\label{1.2.4}
The pair $(\fg,K)$ is trivial. In other words, $\fg$ is the zero Lie algebra $0$, and $K$ is the trivial group scheme $\Spec k$. In this case, the functor $I^{0,\Spec k}_{\fq,M}$ will be denoted by $H^0(\fq,M,-)$.
\end{ex}
\begin{ex}[The algebraic Borel-Weil induction]\label{1.2.5}
Let $G$ be a split reductive group over $\bZ$. Fix a maximal split torus $T$ of $G$, and a positive root system of the Lie algebra $\fg$ of $G$. Write $\bar{\fb}$ for the Lie subalgebra of $\fg$ corresponding to the negative roots. Then we have a map $(\bar{\fb},T)\to(\fg,G)$ of pairs. The corresponding functor $I^{\fg,G}_{\bar{\fb},T}$ is called the Borel-Weil induction. Its derived functor is called the Borel-Weil-Bott induction.
\end{ex}
We also have its derived version:
\begin{note}\label{1.2.6}
Let $(\fg,K)$ be a pair. Then denote the unbounded derived category of $(\fg,K)$-modules and its full subcategory spanned by complexes cohomologically bounded below by $D(\fg,K)$ and $D^+(\fg,K)$ respectively.
\end{note}
\begin{ThmD}[Theorem \ref{3.1.10}]\label{D}
Let $k\to k'$ be a flat ring homomorphism, and $(\fq,M)\to(\fg,K)$ be a map of pairs. Suppose that the following conditions are satisfied:
\begin{enumerate}
\renewcommand{\labelenumi}{(\roman{enumi})}
\item $\fk\oplus\fq\to\fg$ is surjective.
\item $\fq$ and $\fg$ are finitely generated as $k$-modules.
\end{enumerate}
Then we have a natural isomorphism
\[(\bR I^{\fg,K}_{\fq,M}-)\otimes_k k'\simeq \bR I^{\fg\otimes k',K\otimes k'}_{\fq\otimes k',M\otimes k'}(-\otimes_k k')\]
on $D^+(\fq,M)$.
\end{ThmD}
In view of Theorem \ref{D}, the cohomology modules of $I^{\fg,K}_{\fq,M}$ over $\bZ$ are $\bZ$-forms of those over $\bQ$ via the base change $\otimes\bQ$ under the suitable conditions. As mentioned in the introduction of \cite{MR4007195}, it is an expected new phenomenon that the cohomology involve torsions. We give an example in \cite{1712.07336}.

It will be convenient to consider an unbounded analog of Theorem \ref{D}. In fact, then we can use infinite homotopy colimits. They are needed when we consider the homotopy descent for instance (\cite{1001.1556}). The idea of descent and its applications to number theory have already appeared in \cite{MR3770183}. For the proof of Theorem \ref{D}, we do not have a standard resolution. Instead we prove that $\otimes k'$ sends injective objects to acyclic objects with respect to $I^{\fg\otimes k',K\otimes k'}_{\fq\otimes k',M\otimes k'}$. Then we see the base change formula of complexes degreewise. Therefore the argument does not extend literally to the unbounded case. To establish an unbounded analog, we replace the unbounded derived categories. For a pair $(\fg,K)$ over a Noetherian ring $k$, write the stable derived category of $(\fg,K)$-modules by $\Ind\Coh(\fg,K)$ in the sense of \cite{MR2157133}. In terms of higher categories, this can be thought of as the ind-completion (see \cite{MR2522659}) of the $\infty$-category $\Coh(\fg,K)$ of cohomologically bounded complexes whose cohomologies are finitely generated as $(\fg,K)$-modules.

Let $k\to k'$ be a flat homomorphism of Noetherian rings, and $(\fq,M)\to (\fg,K)$ be a map of pairs over $k$. Then we can define the ind-analogs of the functors above:
\[-\otimes k':\Ind\Coh(\fg,K)\to \Ind\Coh(\fg\otimes k',K\otimes k')\]
\[-\otimes k':\Ind\Coh(\fq,M)\to \Ind\Coh(\fq\otimes k',M\otimes k')\]
\[I_{\fq,M}^{\fg,K,\ind}:\Ind\Coh(\fq,M)\to \Ind\Coh(\fg,K)\]
\[I_{\fq\otimes k',M\otimes k'}^{\fg\otimes k',K\otimes k',\ind}:\Ind\Coh(\fq\otimes k',M\otimes k')\to \Ind\Coh(\fg\otimes k',K\otimes k').\]
\begin{ThmE}[Theorem \ref{3.3.4}]\label{E}
There is a canonical isomorphism
\[I^{\fg,K,\ind}_{\fq,M}(-)\otimes k'\to I^{\fg\otimes k',K\otimes k',\ind}_{\fq\otimes k',M\otimes k'}(-\otimes k').\]
Moreover, it restricts to the natural isomorphism $\bR I^{\fg,K}_{\fq,M}(-)\otimes k'\simeq \bR I^{\fg\otimes k',K\otimes k'}_{\fq\otimes k',M\otimes k'}(-\otimes k')$ of Theorem \ref{D} under the identifications
\[\Ind\Coh(\fq,M)^+\simeq D(\fq,M)^+\]
\[\Ind\Coh(\fg\otimes k',K\otimes k')^+\simeq D(\fg\otimes k',K\otimes k')^+\]
\end{ThmE}
This reduces to a base change formula for $D(\fg,K)$ in special cases by the following assertion:
\begin{PropF}[Proposition \ref{3.3.5}]\label{F}
Suppose that $k$ is a field of characteristic 0, $(\fg,K)$ be a pair with $K$ reductive and $\dim \fg<+\infty$. Then the embedding $\Coh(\fg,K)\to D(\fg,K)$ induces an equivalence $\Ind\Coh(\fg,K)\simeq D(\fg,K)$.
\end{PropF}
We also show a finite analog of Theorem \ref{B}, Theorem \ref{C}, and Theorem \ref{D} without assuming conditions (i) and (ii). This is rather a straightforward generalization of \cite{MR3770183} Corollary 2.2 and Theorem 2.5.
\begin{VarG}[Variant \ref{3.2.12}, Variant \ref{3.2.13}, Lemma \ref{3.2.15}]\label{G}
Let $(\fq,M)\to (\fg,K)$ be a map of pairs over a commutative ring $k$, and $k\to k'$ be a ring homomorphism. Assume that $k'$ is finitely generated and projective as a $k$-module.
\begin{enumerate}
\renewcommand{\labelenumi}{(\arabic{enumi})}
\item There is a canonical isomorphism $\Hom_{\fg,K}(-,-)\otimes k'\cong \Hom_{\fg\otimes k',K\otimes k'}(-\otimes k',-\otimes k')$ on $(\fg,K)\cmod^{op}\times(\fg,K)\cmod$. Here $(\fg,K)\cmod^{op}$ denotes the opposite category to $(\fg,K)\cmod$.
\item There is a natural isomorphism
\[(I^{\fg,K}_{\fq,M}-)\otimes_k k'\cong I^{\fg\otimes_k k',K\otimes_k k'}_{\fq\otimes_k k',M\otimes_k k'}(-\otimes_k k').\]
\item There is a natural isomorphism of the functors on the unbounded derived category of $(\fq,M)$-modules:
\[\bR I^{\fg,K}_{\fq,M}(-)\otimes k'\simeq\bR I^{\fg,K}_{\fq,M}(-\otimes k').\]
\end{enumerate}
\end{VarG}
A typical application is to add $\sqrt{-1}$ to the given ring. In fact, we will need $\sqrt{-1}$ (and other fractions) to make integral forms of compact Lie groups to be split. For instance, the special orthogonal group $\SO(2)\cong\Spec \bZ\left[x,y\right]/(x^2+y^2-1)$ is isomorphic to the split torus of rank 1 after the base change to $\bZ\left[\sqrt{-1},\frac{1}{2}\right]$.
\begin{note}\label{1.2.7}
Let $(\fh,K)\to(\fg,K)$ be a map of pairs over a commutative ring with $K\to K$ being the identity. The left and right adjoint functors to the forgetful functor from the category of $(\fg,K)$-modules to that of $(\fh,K)$-modules will be denoted by $\ind^\fg_\fh$ and $\pro^\fg_\fh$ respectively.
\end{note}
Our strategy of the proofs of Theorem \ref{B} and Variant \ref{G} is to use general arguments on generators to reduce them to the group case through the induction $\ind^\fg_\fk:K\cmod\to(\fg,K)\cmod$. The remaining assertion is then a version of \cite{MR2015057} I.2.10 for flat affine group schemes. Theorem \ref{C} is basically obtained by formal arguments of adjunctions. Remark that we have to analyze the resulting bijections since the inverse map of Theorem \ref{B} is not canonical.

Finally, we discuss flat base changes of $\pro$. Unlike the case $k=\bC$, it should be difficult in general since the internal Hom of $K\cmod$ is quite complicated. In this paper, we find a practically nice setting to imitate the description of \cite{MR1330919} Proposition 5.96. Let $G$ be a real reductive group, $(\fg_\bC,K_\bC)$ be the associated pair over $\bC$ to $G$, and $(\fq_\bC,(K_L)_\bC)$ be a $\theta$-stable parabolic subpair, where $\theta$ is the Cartan involution. Let $\bar{\fu}_\bC$ be the opposite nilradical to $\fq_\bC$. Write $h$ for the element of the Cartan subalgebra corresponding to the half sum of roots of the nilradical $\fu_\bC$ of $\fq_\bC$ (\cite{MR1330919} Proposition 4.70).

Let $k$ be a Noetherian subring of $\bC$, and $(\fq,K_L)\subset(\fg,K)$ be a $k$-form of $(\fq_\bC,(K_L)_\bC)\subset(\fg_\bC,K_\bC)$. Assume that there is a complementary $K_L$-stable subalgebra $\bar{\fu}\subset\fg$ to $\fq$ which is a $k$-form of $\bar{\fu}_\bC$. Moreover, suppose that the following conditions are satisfied:
\begin{enumerate}
\renewcommand{\labelenumi}{(\roman{enumi})}
\item There is a free basis of $\fq$.
\item There is a free basis $\{E_{\alpha_i}\}$ of $\bar{\fu}$ consisting of root vectors of $\bar{\fu}_\bC$.
\item The $(K_L)_\bC$-orbit of $h$ is contained in the Cartan subalgebra. This is satisfied when the Levi subgroup of $G$ corresponding to $(\fq_\bC,(K_L)_\bC)$ belongs to the Harish-Chandra class in the sense of \cite{MR1330919} Definition 4.29.
\end{enumerate}
\begin{ThmH}[Proposition \ref{4.1.3}, Proposition \ref{4.2.2}, Proposition \ref{4.2.5}]\label{H}
Let $Z$ be a torsion-free $(\fq,K_L)$-module. Moreover, assume that $Z\otimes\bC$ is admissible and that $h$ acts on it as a scalar.
\begin{enumerate}
\renewcommand{\labelenumi}{(\arabic{enumi})}
\item The enveloping algebra $U(\bar{\fu})$ is decomposed into a direct sum $U(\bar{\fu})=\oplus_\cO U(\bar{\fu})_\cO$ of $K_L$-submodules $U(\bar{\fu})_\cO$ which are free of finite rank as $k$-modules.
\item There is an isomorphism of $K_L$-modules
\[\pro^\fg_\fq(Z)\cong\oplus\Hom_k(U(\bar{\fu})_\cO,Z).\]
In particular, it enjoys the base change formula
\[\pro^\fg_\fq(Z)\otimes\bC\cong\pro^{\fg_\bC}_{\fq_\bC}(Z\otimes\bC).\]
\end{enumerate}
\end{ThmH}
Suppose that we have a semidirect product $\fq=\fl\oplus\fu$ which is compatible with the Levi decomposition $\fq_\bC=\fl_\bC\oplus\fu_\bC$. Assume also that $\fu$ is free of rank $r<\infty$. For an $(\fl,K_L)$-module $\lambda$ on $k$, (temporarily) define $A_\fq(\lambda)$ as $R^{\dim(\fu_\bC\cap \fk_\bC)}\Gamma\pro^\fg_\fq(\lambda\otimes\wedge^r\fu)$. Then we obtain the base change formula of $A_\fq(\lambda)$ along $k\to\bC$ by combining Theorem \ref{H} and Theorem \ref{D}.
\subsection{Notations}
For a coalgebra $C$ over a commutative ring $k$, $C\comod$ denotes the category of $C$-comodules.

For an integral domain $k$, its fractional field will be denoted by $\Frac(k)$.
\renewcommand{\abstractname}{Acknowledgements}
\begin{abstract}
First of all, the author is greatly indebted to his advisor Professor Hisayosi Matumoto for helpful discussions, advice, careful reading of drafts of this paper, and his patience.

Thanks to Teruhisa Koshikawa and Toshihisa Kubo for the question on the relation of the functor $I^{\fg,K}_{\fq,M}$ and the base change along $\bZ\to\bC$ which motivates him to start the studies in this paper.

He found the most results of this paper during his stay in Karlsruhe. He thanks Fabian Januszewski, Liang-Wen Lin, Martin Bari\v{c}, Sven Caspart, Tobias Columbus, Frank Herrlich, Christoph Karg, Stefan K\"uhnlein, Verena M\"ohler, Benjamin Peters, Felix Wellen, and all other people he met there for helping and supporting him. He is grateful to Fabian Januszewski for stimulating discussions, and comments and advice about further directions from his papers \cite{MR4007195} and \cite{1606.04320}. 

This work was supported by JSPS Kakenhi Grant Number 921549 and the Program for Leading Graduate Schools, MEXT, Japan.
\end{abstract}
\section{Comodules}
\subsection{Generalities on comodules}
In this section, let $C$ be a coalgebra over a commutative ring $k$. It is easy to formulate the base change adjunction of comodules. Namely, for a $k$-algebra $k'$, the forgetful functor and the base change give rise to the natural bijection
\[\Hom_C(V,W)\cong\Hom_{C\otimes k'}(V\otimes k',W),\]
where $V$ is a $C$-comodule, and $W$ is a $C\otimes k'$-comodule.

In the rest, assume that $C$ is flat over $k$. We note general constructions of comodules. Let $V$ be a $C$-comodule, $V_0$ be a $k$-submodule, and $S$ be a subset of $V$.
\begin{cons}\label{2.1.1}
Define $\cI_{V,V_0}$ as the full subcategory of the overcategory $C\comod_{/V}$ spanned by subcomodules of $V$ contained in $V_0$, and $V_0^\circ$ be the colimit of the canonical functor $\cI_{V,V_0}\to C\comod$.
\end{cons}
\begin{prop}\label{2.1.2}
\begin{enumerate}
\renewcommand{\labelenumi}{(\arabic{enumi})}
\item The category $\cI_{V,V_0}$ is filtered.
\item The comodule $V_0^\circ$ exhibits the maximal subcomodule of $V$ contained in $V_0$.
\end{enumerate}
\end{prop}
\begin{proof}
To prove (1), suppose that we are given two comodules $W,W'\subset V_0$. Then the image of the sum $W\oplus W'\to V$ belongs to $\cI_{V,V_0}$. Since $\cI_{V,V_0}$ is a diagram of subobjects of a fixed object of a category, the other condition automatically follows. Part (2) now follows since filtered colimits of $k$-modules are exact.
\end{proof}
\begin{cons}\label{2.1.3}
Define $\cJ_{V,S}$ as the full subcategory of the overcategory $C\comod_{/V}$ spanned by subcomodules of $V$ containing $S$, and set $\langle S\rangle$ as the limit of the canonical diagram $\cJ_{V,S}\to C\comod$. 
\end{cons}
\begin{prop}\label{2.1.4}
The comodule $\langle S\rangle$ exhibits the smallest subcomodule of $V$ containing $S$.
\end{prop}
\begin{proof}
Choose a vertex $S\subset W\subset V$ of $C\comod_{/V}$, and denote the composite arrow $\langle S\rangle\to W\to V$ by $i$. Observe that $i$ is independent of the choice of $W$. In fact, take another object $S\subset W'\subset V$. Since monomorphisms are stable under pullbacks, $W\times_V W'$ is a subcomodule of $W,W'$ containing $S$. The resulting commutative diagram
\[\xymatrix{&W\ar[rd]&\\
\langle S\rangle\ar[r]\ar[ru]\ar[rd]&W\times_V W'\ar[r]\ar[u]\ar[d]&V\\
&W'.\ar[ru]}\]
shows the independence.

We next prove that the map $\langle S\rangle\to V$ is a monomorphism. Suppose that we are given two homomorphisms $U\overset{f}{\underset{g}{\rightrightarrows}}\langle S\rangle\overset{i}{\to}V$ such that $i\circ f=i\circ g$. Let us denote the canonical projection $\langle S\rangle\to W$ by $p_W$, and the inculsion $W\hookrightarrow V$ by $i_W$. The equality
\[i_W\circ p_W\circ f=i\circ f=i\circ g=i_W\circ p_W\circ g\]
implies $p_W\circ f=p_W\circ g$. Therefore these equal maps form a cone over $\cJ_{V,S}$ whose vertex is $U$. Moreover, the two maps $U\overset{f}{\underset{g}{\rightrightarrows}}\langle S\rangle$ are morphisms of cones. Since $\langle S\rangle$ is terminal, the two arrows are equal.

Finally, we prove that $\langle S\rangle$ is the minimum. In fact, if we are given a subcomodule $S\subset W\subset V$, then we have a commutative diagram
\[\xymatrix{\langle S\rangle\ar[rd]\ar[rr]&&V\\
&W\ar[ru]}\]
by definition. Since the upper horizontal and the upper right diagonal arrows are injective, so is the rest.
\end{proof}
\begin{ex}[\cite{MR499562} 27 Exercise 8]\label{2.1.5}
Set $C$ as the coordinate ring of the affine group scheme $\SL_2$ over $\bZ$. Let $V$ be an irreducible representation of $\SL_2$ over $\bQ$ with $\dim V=n+1$, and $v_n$ be a highest weight vector of $V$. Then $V^m:=\langle v_n\rangle\subset V$ is described as follows:
\[V^m=\oplus_{i=0}^{n}\bZ v_{n-2i}\]
\[Ev_{n-2i}=(n-i+1)v_{n-2i+2}\]
\[Fv_{n-2i}=(i+1)v_{n-2i-2}.\]
\end{ex}
\begin{prop}[\cite{MR2015057} I.2.13]\label{2.1.6}
For an element $v\in V$, the comodule $\langle v\rangle$ is contained in a finitely generated $k$-module.
\end{prop}
This leads us to a categorical conclusion for comodules. To state it, we prepare some general teminologies and facts. For our applications, we may restrict ourselves to abelian categories if necessary. For general references, see \cite{MR1294136} Section 1, \cite{MR1291599} Section 4, and \cite{MR1313497} Section 1, 5.
\begin{defn}\label{2.1.7}
Let $\cA$ be a locally small cocomplete abelian category. Then a small set $G$ of objects of $\cA$ is called a family of generators if the following equivalent conditions are satisfied:
\begin{enumerate}
\renewcommand{\labelenumi}{(\alph{enumi})}
\item Maps $f,g:X\to Y$ satisfying $f\circ e=g\circ e$ for any $Q\in G$ and $e\in\Hom(Q,X)$ are equal.
\item For every object $X\in\cA$, the morphism $\coprod_{\stackrel{Q\in G}{e\in\Hom(Q,X)}}Q\stackrel{\coprod e}{\to}X$
is epic.
\item If we are given a monomorphism $i:X\to Y$ which is not an isomorphism, there exists a map $Q\to Y$ with $Q\in G$ that does not factors through $i$.
\item A morphism $X\to Y$ in $\cA$ is an isomorphism if and only if for any member $Q\in G$, the induced map $\Hom(Q,X)\to\Hom(Q,Y)$ is a bijection.
\end{enumerate}
\end{defn}
The next fact is obvious by definition:
\begin{lem}\label{2.1.8}
A functor between locally small (cocomplete abelian) categories with a faithful right adjoint functor respects families of generators.
\end{lem}
\begin{defn}\label{2.1.9}
Let $\cC$ be a locally small category with small filtered colimits. Then an object $A\in\cC$ is said to be compact if for any small filtered diagram $Y_\bullet$ of $\cC$, the induced map
\[\varinjlim\Hom(A,Y_\bullet)\to\Hom(A,\varinjlim Y_\bullet)\]
is a bijection.
\end{defn}
\begin{defn}\label{2.1.10}
A locally small cocomplete abelian category is compactly generated if it admits a small set of compact generators.
\end{defn}
We have a nontrivial consequence from characterizations of compactly generated categories:
\begin{lem}[\cite{MR1294136} Remark 1.9, the proof of Theorem 1.11]\label{2.1.11}
Let $\cA$ be a compactly generated (abelian) category with a small set $G$ of compact generators. Then compact objects of $\cA$ are generated by $G$ under finite colimits.
\end{lem}
These are used in 3.2 and 4.1 as key techniques. We now go back to comodules.
\begin{cor}\label{2.1.12}
If $k$ is Noetherian, the category $C\comod$ is compactly generated. In other words, every comodule is the union of its finitely generated subcomodules.
\end{cor}
\begin{proof}
The assertions follow from Proposition \ref{2.1.6}. Note that for a $C$-comodule $V$, the following conditions are equivalent (\cite{MR2066503} Proposition 1.3.3):
\begin{enumerate}
\renewcommand{\labelenumi}{(\alph{enumi})}
\item $V$ is compact in $C\comod$;
\item $V$ is compact as a $k$-module;
\item $V$ is a finitely presented $k$-module.
\end{enumerate}
\end{proof}
\begin{cor}\label{2.1.13}
Suppose that $k$ is a PID. Then indecomposable comodules which are free of finite rank over $k$ form a family of generators of $C\comod$.
\end{cor}
\begin{proof}
According to the proof of \cite{MR2138086} Proposition 1.2, subcomodules of direct sums of finite copies of $C$ form a family of generators of $C\comod$. In view of Proposition \ref{2.1.6}, we may restrict the members of the family to torsion-free finitely generated subcomodules. The assertion is now obvious.
\end{proof}
The next lemma is used in the end of this paper:
\begin{lem}\label{2.1.14}
Let $k\to k'$ be an injective homomorphism of commutative rings, $C$ be a flat coalgebra over $k$, and $V$ be a $C$-comodule. Suppose that the following conditions are satisfied:
\begin{enumerate}
\renewcommand{\labelenumi}{(\roman{enumi})}
\item $V$ is flat as a $k$-module.
\item There is a decomposition $V\cong \oplus_\cO V_\cO$ as a $k$-module.
\item $V\otimes k'\cong \oplus_\cO V_\cO\otimes k'$ is a direct sum of $C\otimes k'$-subcomodules of $V\otimes k'$.
\end{enumerate}
Then each of $V_\cO$ is a subcomodule of $V$, and $V\cong \oplus_\cO V_\cO$ exhibits a decomposition as a $C$-comodule.
\end{lem}
\begin{proof}
According to (ii), we have an isomorphism $V\otimes C\cong \oplus_\cO V_\cO\otimes C$. It will suffice to show that the coaction respects each $\cO$-component. Take the base change along $k\to k'$ to obtain a commutative diagram
\[\xymatrix{V\ar[r]\ar[d]&V\otimes C\ar[d]\\
V\otimes k'\ar[r]&(V\otimes k')\otimes_{k'}(C\otimes k').}\]
Since $V$ and $C$ are flat, the vertical arrows are injective. Therefore the assertion is reduced to $k=k'$, and it is equivalent to (iii).
\end{proof}
\subsection{Representations of flat affine group schemes and $(\fg,K)$-modules}
Let $H$ be a commutative Hopf algebra, and write $K=\Spec H$. For a $k$-module $V$ and a $k$-algebra $R$, set $\Aut_R(V\otimes R)$ as the group of automorphisms of the $R$-module $V\otimes R$. This determines a group $k$-functor $\Aut(V):\CAlg_k\to\Grp;R\mapsto\Aut_R(V\otimes R)$, where $\CAlg_k$ (resp.\ $\Grp$) is the category of commutative $k$-algebras (resp.\ groups). We also write $\CAlg_{k,{\rm flat}}$ for the full subcategory of $\CAlg_k$ spanned by flat $k$-algebras. Note that $\CAlg_{k,{\rm flat}}$ is stable under $\otimes$. Recall that a representation of $K$ is a $k$-module $V$, equipped with a homomorphism $K\to\Aut(V)$ of group $k$-functors. Equivalently, a representation is a $k$-module, equipped with an $R$-linear group action of $K(R)$ on $V\otimes R$ for each $k$-algebra $R$ such that for $f:R\to R'$ and $g\in K(R)$ the diagram
\[\xymatrix{V\otimes R\ar[r]^{\nu(g)}\ar[d]_f&V\otimes R\ar[d]^f\\
V\otimes R'\ar[r]_{\nu(f\circ g)}&V\otimes R'
}\]
commutes. A $k$-module homomorphism $f:V\to V'$ of representations of $K$ is said to be a $K$-homomorphism if for all $k$-algebras $R$, the diagrams
\[\xymatrix{
V\otimes R\ar[r]^{\nu(g)}\ar[d]_{f\otimes id_R}&V\otimes R\ar[d]^{f\otimes id_R}\\
V'\otimes R\ar[r]_{\nu'(g)}&V'\otimes R
}\]
commute. Set $K\cmod$ as the category of representations of $K$. If $K$ is flat over $k$, define $K\cmod_{\rm flat}$ in a similar way.
\begin{lem}\label{2.2.1}
The categories $K\cmod$ and $H\comod$ are isomorphic. Moreover, if $K$ is flat, these are also isomorphic to $K\cmod_{\rm flat}$.
\end{lem}
\begin{proof}
See \cite{MR547117} Theorem 3.2 for the first assertion. In view of its proof, the coaction of $H$ is recovered by the actions of the valued point groups $K(k)$, $K(H)$, and $K(H\otimes H)$. Therefore the same argument proves $H\comod\cong K\cmod_{\rm flat}$ if $H$ is flat.
\end{proof}
Suppose next that $H$ is flat over $k$. Though we have a general description of the internal Hom of the symmetric monoidal category $H\comod$ (\cite{MR2066503} Theorem 1.3.1), it is usually too complicated to compute in practice. Here we give a better realization in a special case:
\begin{prop}\label{2.2.2}
Let $K$ be an affine group scheme, and $V,V'$ be $K$-modules.
\begin{enumerate}
\renewcommand{\labelenumi}{(\arabic{enumi})}
\item If $V$ is finitely generated and projective as a $k$-module, there is a natural $K$-action on $\Hom(V,V')$. Moreover, the standard adjunction
$\Hom_k(-\otimes V,V')\cong\Hom_k(-,\Hom(V,V'))$
restricts to
\[\Hom_K(-\otimes V,V')\cong\Hom_K(-,\Hom(V,V')).\]
\item Suppose that $K$ is flat. If $V$ satisfies Condition \ref{1.1.4} (2), there is a natural $K$-action on $\Hom(V,V')$. Moreover, the standard adjunction
$\Hom_k(-\otimes V,V')\cong\Hom_k(-,\Hom(V,V'))$
restricts to
\[\Hom_K(-\otimes V,V')\cong\Hom_K(-,\Hom(V,V')).\]
\end{enumerate}
\end{prop}
\begin{proof}
Suppose that $V$ is finitely generated and projective as a $k$-module. Then for any $k$-algebra $R$, we have a canonical isomorphism $\Hom(V,V')\otimes R\cong\Hom(V,V'\otimes R)\cong\Hom_R(V\otimes R,V'\otimes R)$. Under this identification, we put a $K(R)$-action on $\Hom(V,V')\otimes R$ by
\[(\nu(g)f)(v)=\nu_{V'}(g)f(\nu_V(g^{-1})v).\]
Running through all $R$, we obtain $\Hom(V,V')\in K\cmod$. We can see the adjunction in the usual way. If $K$ is flat, we may restrict $R$ to be flat (Lemma \ref{2.2.1}). Then the same argument works for $V$ with Condition \ref{1.1.4} (2).
\end{proof}
\begin{proof}[Proof of Lemma \ref{1.1.9}]
This is proved in a similar way to \cite{MR4007195} Theorem 2.3.6 (2). The adjoint representation of $K$ makes sense from Lemma \ref{2.2.1} (see also loc.\ cit.\ below Condition 2.2.2). The construction of the colocalization $(-)^\fk$ in loc.\ cit.\ Lemma 2.3.2 works from Proposition \ref{2.2.2} (2).
\end{proof}
\begin{prop}\label{2.2.3}
Let $(\fg,K)$ be a pair over a commutative ring $k$.
\begin{enumerate}
\renewcommand{\labelenumi}{(\arabic{enumi})}
\item For $(\fg,K)$-modules $V$ and $V'$, the tensor product $V\otimes W$ is a $(\fg,K)$-module for the tensor representation of $K$ and
\[\pi_{V\otimes V'}(x)(v\otimes v')=\pi_V(x)v\otimes v'+v\otimes \pi_{V'}(x)v',\]
where $v\otimes v'\in V\otimes V'$, $x\in\fg$, and $\pi_V$ (resp.\ $\pi_{V'}$) denotes the action of $\fg$ on $V$ (resp.\ $V'$).
\item The category $(\fg,K)\cmod$ is closed symmetric monoidal for the monoidal structure defined in (1). Moreover, the closed structure is compatible with that of $K\cmod$.
\end{enumerate}
\end{prop}
\begin{note}\label{2.2.4}
The internal Hom of the symmetric monoidal category $K\cmod$ will be denoted by $F(-,-)$.
\end{note}
\begin{proof}[Proof of Proposition \ref{2.2.3}]
It is easy to see that the tensor product $V\otimes V'$ of (1) is a module over both $K$ and $\fg$. Apply $-\otimes V'$ and $-\otimes V$ to the $K$-equivariant maps
\[\fg\otimes V\to V\]
\[\fg\otimes V'\to V'\]
respectively. Since $K\cmod$ is symmetric monoidal, we have two $K$-equivariant maps from $\fg\otimes V\otimes V'$ to $V\otimes V'$. Since their sum coincides with $\pi_{V\otimes V'}$, the latter is also $K$-equivariant. The actions of $\fk$ coincide by the Leibnitz rule of the differential representations for the tensor product. To see that this defines a symmetric monoidal category, it will suffice to show that the constraints of associativity and symmetry of $K\cmod$ respect the $\fg$-actions. This is obvious.

Recall that we have a $K$-equivariant $k$-homomorphism $\fg\to U(\fg)\otimes U(\fg);x\mapsto x\otimes 1-1\otimes x$ (regard $\fg\cong \fg\otimes k\cong k\otimes \fg$). For $(\fg,K)$-modules $V,V'$, define $\pi=\pi_{F(V,V')}:\fg\otimes F(V,V')\to F(V,V')$ by the following composite arrows:
\[\begin{split}
\fg\otimes F(V,V')\otimes V
&\to U(\fg)\otimes U(\fg)\otimes F(V,V')\otimes V\\
&\cong U(\fg)\otimes F(V,V')\otimes U(\fg)\otimes V\\
&\overset{\pi_V}{\to}U(\fg)\otimes F(V,V')\otimes V\\
&\to U(\fg)\otimes V'\\
&\overset{\pi_{V'}}{\to} V'.
\end{split}\]
This is $K$-equivariant by definition. To see that this is a $\fg$-action, we see that the two maps
\[\fg\otimes\fg\otimes F(V,V')\rightrightarrows F(V,V')\]
coincide. If we write $f\otimes v\mapsto f(v)$ for the counit $F(V,V')\otimes V\to V'$, it is equivalent to
\[(\pi(\left[x,y\right])f)(v)=(\pi(x)(\pi(y)f))(v)-(\pi(y)(\pi(x)f))(v)\]
for $x,y\in\fg$ and $f\in F(V,V')$. Observe that $\pi_{F(V,V')}$ is characterized by the equality \[(\pi(x)f)(v)=\pi_{V'}(x)f(v)-f(\pi_V(x)v)\]
by definition, and thus
\[\begin{split}
(\pi(x)(\pi(y)f))(v)
&=\pi(x)(\pi(y)f)(v)-(\pi(y)f)(\pi(x)v)\\
&=\pi(x)\pi(y)f(v)-\pi(x)f(\pi(y)v)-\pi(y)f(\pi(x)v)+f(\pi(y)\pi(x)v).
\end{split}\]
The assertion now follows from the formal computation
\[\begin{split}
(\pi(\left[x,y\right])f)(v)
&=\pi(\left[x,y\right])f(v)-f(\pi(\left[x,y\right])v)\\
&=\pi(x)\pi(y)f(v)-\pi(y)\pi(x)f(v)-f(\pi(x)\pi(y)v)+f(\pi(y)\pi(x)v)\\
&=(\pi(x)(\pi(y)f))(v)-(\pi(y)(\pi(x)f))(v).
\end{split}\]

We next show that $F(V,V')$ is a $(\fg,K)$-module. Since $V,V'$ are $(\fg,K)$-modules, the action $\pi$ can be rewritten as
\[(\pi(\xi)f)(v)=\pi(\xi)f(v)-f(\pi(\xi)v)=d\nu(\xi)f(v)-f(d\nu(\xi)v).\]
for $\xi\in\fk$. Since the counit $F(V,V')\otimes V\to V'$ is $\fk$-equivariant with respect to the differential representations, we have
\[(\pi(\xi)f)(v)=d\nu(\xi)f(v)-f(d\nu(\xi)v)=(d\nu(\xi) f)(v).\]

Finally, we prove that $F(V,V')$ exhibits the closed structure. Let $V''$ be another $(\fg,K)$-module, and $\varphi:V''\to F(V,V')$ be a $K$-module homomorphism. It will suffice to show that $\varphi$ is $\fg$-equivariant if and only if the composition $\Phi:V''\otimes V\to F(V,V')\otimes V\to V'$ is $\fg$-equivariant. Observe that the following conditions are equivalent:
\begin{enumerate}
\renewcommand{\labelenumi}{(\alph{enumi})}
\item $\varphi$ is $\fg$-equivariant;
\item The diagram 
\[\xymatrix{
\fg\otimes V''\ar[r]^{id_\fg\otimes\varphi}\ar[d]_{\pi_{V''}}
&\fg\otimes F(V,V')\ar[d]^{\pi_{F(V,V')}}\\
V''\ar[r]&F(V,V')
}\]
commutes;
\item The diagram 
\[\xymatrix{
\fg\otimes V''\otimes V\ar[r]^\varphi\ar[dd]_{\pi_{V''}}
&\fg\otimes F(V,V')\otimes V\ar[rd]^{\pi_{F(V,V')}}\ar[dd]\\
&&F(V,V')\otimes V\ar[ld]\\
V''\otimes V\ar[r]_\Phi&V'
}\]
commutes.
\end{enumerate}
One can also rewrite (c) as
\[\Phi(\pi(x)v''\otimes v)=\pi(x)\varphi(v'')(v)-\varphi(v'')(\pi(x)v)
=\pi(x)\Phi(v''\otimes v)-\Phi(v''\otimes \pi(x)v)\]
which is equivalent to saying that $\Phi$ is $\fg$-equivariant. This completes the proof.
\end{proof}
\begin{defn}\label{2.2.5}
Let $(\fg,K)$ be a pair over a commutative ring $k$, and $V$ be a $(\fg,K)$-module. Then set $V^c=F(V,k)$.
\end{defn}
Suppose that we are given a map $(\fq,K)\to(\fg,K)$ of pairs which is the identity on $K$.
\begin{cor}\label{2.2.6}
For a $(\fq,K)$-module $W$ and a $(\fg,K)$-module $V$, there is a natural isomorphism
$\ind^\fg_\fq W\otimes V\cong\ind^\fg_\fq(W\otimes\cF^{\fq,K}_{\fg,K} (V))$. 
\end{cor}
\begin{proof}
For a $(\fg,K)$-module $X$, we have a natural bijection
\[\begin{split}
\Hom_{\fg,K}(\ind^\fg_\fq W\otimes V,X)
&\cong\Hom_{\fg,K}(\ind^\fg_\fq W,F(V,X))\\
&\cong\Hom_{\fq,K}(W,\cF^{\fq,K}_{\fg,K}(F(V,X)))\\
&\cong\Hom_{\fq,K}(W\otimes \cF^{\fq,K}_{\fg,K}(V),\cF^{\fq,K}_{\fg,K}(X))\\
&\cong\Hom_{\fg,K}(\ind^\fg_\fq(W\otimes \cF^{\fq,K}_{\fg,K} (V)),X).
\end{split}\]
The assertion now follows from the Yoneda lemma.
\end{proof}
\begin{cor}[The easy duality]\label{2.2.9}
There is a natural isomorphism $\ind^\fg_\fq(W)^c\cong\pro^\fg_\fq(W^c)$ for a $(\fq,K)$-module $W$.
\end{cor}
\begin{proof}
For a $(\fg,K)$-module $V$, we have
\[\begin{split}
\Hom_{\fg,K}(V,\ind^\fg_\fq(W)^c)
&\cong\Hom_{\fg,K}(V\otimes\ind^\fg_\fq W,k)\\
&\cong\Hom_{\fg,K}(\ind^\fg_\fq(V\otimes W),k)\\
&\cong\Hom_{\fq,K}(V\otimes W,k)\\
&\cong\Hom_{\fq,K}(V,W^c)\\
&\cong\Hom_{\fg,K}(V,\pro^\fg_\fk(W^c)).
\end{split}\]
The assertion now follows from the Yoneda Lemma.
\end{proof}
\section{The Flat Base Change Theorems}
\subsection{The main statements}
We start with the definition of the base change functor.
Let $k\to k'$ be a homomorphism of commutative rings. For an algebra $\cA$ over $k$, an $\cA\otimes k'$-module $W$ is an $\cA$-module for
\[\cA\otimes_k W\cong(\cA\otimes k')\otimes_{k'}W\to W.\]
Conversely, if we are given an $\cA$-module $V$, $V\otimes k'$ is an $\cA\otimes k'$-module for
\[(\cA\otimes k')\otimes_{k'} (V\otimes k')\cong(\cA\otimes V)\otimes k'\to V\otimes k'.\]
These form an adjunction
\[\Hom_\cA(V,W)\cong\Hom_{\cA\otimes k'}(V\otimes k',W).\]
Similarly, if we are given a flat affine group scheme $K$ over $k$, we have the base change adjunction (see the beginning of Section 2.1). In terms of $k$-functors, they are described as follows:
For $k'$ a $k$-algebra, we have
\[\begin{split}
K(R)&\to(K\otimes k')(R\otimes k')\\
&\to\Aut_{R\otimes k'}(W\otimes_{k'}(R\otimes k'))\\
&\cong\Aut_{R\otimes k'}(W\otimes R)\\
&\to\Aut_R(W\otimes R)
\end{split}\]
\[\begin{split}
(K\otimes k')(R)&=\Hom_{k'}(k\left[K\right]\otimes k',R)\\
&\cong\Hom_k(k\left[K\right],R)\\
&\to\Aut_R(V\otimes R)\\
&\cong\Aut_R((V\otimes k')\otimes_{k'}R).
\end{split}\]
Hence the differential representations are compatible with the restrictions and the flat base changes. That is, let $k'$ be a flat $k$-algebra.
\begin{itemize}
\item If we are given a $K\otimes k'$-module $W$, the differential representation on the restriction of $W$ to $K$ coincides with
\[\fk\otimes W\cong(\fk\otimes k')\otimes_{k'}W\to W;\]
\item For a $K$-module $V$, the differential representation of $K\otimes k'$ on the $K\otimes k'$-module $V\otimes k'$ is induced from
\[\fk\otimes V\to V\to V\otimes k'\]
by the universality of the base change.
\end{itemize}
We now obtain the following consequence from these functorial constructions:
\begin{prop}\label{3.1.1}
Let $(\fg,K)$ be a pair over $k$, and $k'$ be a flat $k$-algebra. Then we have an adjunction
\[-\otimes_k k':(\fg,K)\cmod\leftrightarrows(\fg\otimes k',K\otimes k')\cmod:\Res^k_{k'}.\]
\end{prop}
\begin{rem}\label{3.1.2}
If $K$ is smooth over $k$, the base change makes sense for all $k'$ since the smoothness is stable under arbitrary base changes.
\end{rem}
\begin{rem}\label{3.1.3}
For a weak pair $(\fg,K)$ in the sense of \cite{MR4007195}, the base change of weak $(\fg,K)$-modules always makes sense even if $K$ does not satisfy Condition \ref{1.1.6}.
\end{rem}
\begin{cor}\label{3.1.4}
Let $(\fq,M)\to(\fg,K)$ be a map of pairs over $k$, $k'$ be a flat $k$-algebra, and $V$ be a $(\fq\otimes k',M\otimes k')$-module. Then there is an isomorphism
\[I^{\fg,K}_{\fq,M}(\Res^k_{k'}(V))\cong \Res^k_{k'}(I^{\fg\otimes k',K\otimes k'}_{\fq\otimes k',M\otimes k'}(V)).\]
In particular, if $W$ is a $(\fq,M)$-module,
\[I^{\fg,K}_{\fq,M}(\Res^k_{k'}(W\otimes k'))\cong \Res^k_{k'}(I^{\fg\otimes k',K\otimes k'}_{\fq\otimes k',M\otimes k'}(W\otimes k')).\]
\end{cor}
\begin{proof}
Pass to the right adjoints of $(-\otimes k')\circ \cF^{\fq,M}_{\fg,K}\cong\cF^{\fq\otimes k',M\otimes k'}_{\fg\otimes k',K\otimes k'}\circ (-\otimes k')$: For any $(\fg,K)$-module $X$, we have
\[\begin{split}
\Hom_{\fg,K}(X,I^{\fg,K}_{\fq,M}(\Res^k_{k'}(V)))
&\cong\Hom_{\fq,M}(X,\Res^k_{k'}(V))\\
&\cong\Hom_{\fq\otimes k',M\otimes k'}(X\otimes k',V)\\
&\cong\Hom_{\fg\otimes k',K\otimes k'}(X\otimes k',I^{\fg\otimes k',K\otimes k'}_{\fq\otimes k',M\otimes k'}(V))\\
&\cong\Hom_{\fg,K}(X,\Res^k_{k'}(I^{\fg\otimes k',K\otimes k'}_{\fq\otimes k',M\otimes k'}(V))).
\end{split}\]
The assertion now follows from the Yoneda lemma.
\end{proof}
\begin{cons}[The comparison natural transform]\label{3.1.5}
Let $(\fq,M)\to(\fg,K)$ be a map of pairs over $k$, and $k'$ be a flat $k$-algebra. Then applying $I^{\fg,K}_{\fq,M}$ to the unit of Proposition \ref{3.1.1}, we obtain a natural transform
\[I^{\fg,K}_{\fq,M}(-)\to I^{\fg,K}_{\fq,M}(\Res^k_{k'}(-\otimes k'))
\cong \Res^k_{k'}(I^{\fg\otimes k',K\otimes k'}_{\fq\otimes k',M\otimes k'}(-\otimes k')).\]
Pass to the adjunction of Proposition \ref{3.1.1} to get
\[I^{\fg,K}_{\fq,M}(-)\otimes k'\to I^{\fg\otimes k',K\otimes k'}_{\fq\otimes k',M\otimes k'}(-\otimes k')\]
which will be referred to as $\iota=\iota_{k,k'}$.
\end{cons}
In the rest of this section, assume $k$ to be Noetherian. We will prove in Section 3.2 below:
\begin{thm}[Flat base change theorem]\label{3.1.6}
Let $k'$ be a flat $k$-algebra, and $(\fg,K)$ be a pair over $k$ with $\fg$ finitely generated over $k$. Then for any finitely generated $(\fg,K)$-module $X$, we have an isomorphism
\[\Hom_{\fg,K}(X,-)\otimes k'\cong\Hom_{\fg\otimes k',K\otimes k'}(X\otimes k',-\otimes
 k').\]
\end{thm}
\begin{thm}\label{3.1.7}
Let $k\to k'$ be a flat ring homomorphism, and $(\fq,M)\to(\fg,K)$ be a map of pairs. Suppose that the following conditions are satisfied:
\begin{enumerate}
\renewcommand{\labelenumi}{(\roman{enumi})}
\item $\fk\oplus\fq\to\fg$ is surjective.
\item $\fq$ and $\fg$ are finitely generated as $k$-modules.
\end{enumerate}
Then $\iota:(I^{\fg,K}_{\fq,M}-)\otimes_k k'\to I^{\fg\otimes_k k',K\otimes_k k'}_{\fq\otimes_k k',M\otimes_k k'}(-\otimes_k k')$ (Construction \ref{3.1.5}) is an isomorphism.
\end{thm}
We also have its derived version:
\begin{defn}\label{3.1.8}
Let $k$ be a Noetherian ring, and $(\fg,K)$ be a pair. Suppose that $\fg$ is finitely generated. Set $\Coh(\fg,K)$ as the full subcategory of the derived category $D(\fg,K)$ of $(\fg,K)$-modules spanned by cohomologically bounded complexes with finitely generated cohomologies.
\end{defn}
\begin{thm}\label{3.1.9}
Let $k$ be a Noetherian ring, $(\fg,K)$ be a pair, and $k'$ be a flat $k$-algebra. Then the flat base change theorem
\[\bR\Hom_{\fg,K}(-,-)\otimes k'\simeq\bR\Hom_{\fg,K}(-,-\otimes k')\]
holds on $\Coh(\fg,K)^{op}\times D(\fg,K)^+$.
\end{thm}
\begin{thm}\label{3.1.10}
Let $k\to k'$ be a flat ring homomorphism, and $(\fq,M)\to(\fg,K)$ be a map of pairs. Suppose that the following conditions are satisfied:
\begin{enumerate}
\renewcommand{\labelenumi}{(\roman{enumi})}
\item $\fk\oplus\fq\to\fg$ is surjective.
\item $\fq$ and $\fg$ are finitely generated as $k$-modules.
\end{enumerate}
Then we have a natural isomorphism
\[(\bR I^{\fg,K}_{\fq,M}-)\otimes_k k'\simeq \bR I^{\fg\otimes k',K\otimes k'}_{\fq\otimes k',M\otimes k'}(-\otimes_k k')\]
on $D^+(\fq,M)$.
\end{thm}
For a simple application, we can prove the algebraic Borel-Weil theorem over fields of characteristic 0. Suppose that $k$ is a field of characteristic 0. Let $G$ be a split reductive group. Fix a maximal split torus $T$ of $G$ and a positive root system of the Lie algebra $\fg$ of $G$. Write $\bar{\fb}$ for the Lie subalgebra of $\fg$ corresponding to the negative roots. 
\begin{prop}\label{3.1.11}
Let $\lambda$ be a dominant character of $T$. There is an isomorphism $I^{\fg,G}_{\bar{\fb},T}(\lambda)\otimes\bar{k}\cong
I^{\fg\otimes\bar{k},G\otimes\bar{k}}_{\bar{\fb}\otimes\bar{k},T\otimes\bar{k}}(\lambda\otimes\bar{k})$, where $\bar{k}$ is the algebraic closure of $k$. In particular, $I^{\fg,G}_{\bar{\fb},T}(\lambda)$ is an absolutely irreducible representation of $G$.
\end{prop}
We write $V(\lambda)=I^{\fg,G}_{\bar{\fb},T}(\lambda)$. The coordinate ring of $G$ will be denoted by $\cO(G)$.
\begin{cor}\label{3.1.12}
The homomorphism of coalgebras $\oplus_\lambda\End_k(V(\lambda))\to\cO(G)$ is an isomorphism, where $\lambda$ runs through all dominant characters of $T$.
\end{cor}
\begin{proof}
Passing to the base change along $k\to\bar{k}$, we may assume that $k$ is algebraically closed. Then the assertion follows from the algebraic Peter-Weyl theorem.
\end{proof}
\begin{cor}[\cite{MR0277536}]\label{3.1.13}
The absolutely irreducible representations $V(\lambda)$ of $G$ form a complete list of irreducible representations of $G$.
\end{cor}
\subsection{Proof of the theorems}
In this section, let $k$ be a Noetherian ring.
\begin{lem}\label{3.2.1}
Let $C$ be a flat coalgebra, and $V,X,Y$ be $C$-comodules. Suppose that we are given a commutative diagram of $k$-modules
\[\xymatrix{V\ar[rr]^f\ar[rd]_g&&Y\\
&X\ar[ru]_i,}\]
where $i$ is injective. If the maps $f$ and $i$ intertwine the coactions of $C$, then so does $g$.
\end{lem}
\begin{lem}\label{3.2.2}
Let $(\fg,K)$ be a pair over a commutative ring $k$. Then a $k$-submodule $V'$ of a $(\fg,K)$-module $V$ is a subobject in $(\fg,K)\cmod$ if and only if it is a submodule over both $\fg$ and $K$.
\end{lem}
We omit the proofs of the two above assertions since they follow from standard arguments.
\begin{lem}\label{3.2.3}
Let $\fg$ be a finitely generated Lie algebra over $k$. Then the enveloping algebra $U(\fg)$ is left and right Noetherian.
\end{lem}
\begin{proof}
The assertion follows since the enveloping algebra is by definition a quasi-commutative filtered algebra whose associated graded algebra is generated by $\fg$.
\end{proof}
Recall that a Grothendieck abelian category is said to be locally Noetherian if every object is presented by a filtered colimit of Noetherian objects.
\begin{prop}\label{3.2.4}
Let $k$ be a Noetherian ring, and $(\fg,K)$ be a pair over $k$. If $\fg$ is a finitely generated $k$-module, the category $(\fg,K)\cmod$ is locally Noetherian. Moreover, for a $(\fg,K)$-module $V$, the following conditions are equivalent:
\begin{enumerate}
\renewcommand{\labelenumi}{(\alph{enumi})}
\item $V$ is Noetherian;
\item $V$ is compact;
\item $V$ is finitely generated as a $U(\fg)$-module.
\end{enumerate}
\end{prop}
\begin{proof}
Let $V$ be a $(\fg,K)$-module. From Lemma \ref{3.2.3}, (c) implies (a). Conversely, if $V$ is a Noetherian object, there exists a maximal finitely generated $(\fg,K)$-submodule $V'\subset V$. Choose a finite set $S$ of generators of $V'$. Let $v$ be an arbitrary element of $V$. Then we obtain a $K$-submodule $V_0:=\langle S,v\rangle$ which is finitely generated as a $k$-module (Proposition \ref{2.1.6}). Since the $\fg$-submodule generated by $V_0$ is the image of the map
\[U(\fg)\otimes V_0\to U(\fg)\otimes V\to V,\]
it is a $(\fg,K)$-submodule containing $V'$ (Lemma \ref{3.2.3}). The maximality therefore implies $V=V'$. Hence (c) follows. Moreover, Corollary \ref{2.1.12} then implies that $(\fg,K)\cmod$ is locally Noetherian. The equivalence of (a) and (b) is a consequence of generalities on locally Noetherian abelian categories.
\end{proof}
\begin{defn}\label{3.2.5}
Let $B$ be a bialgebra. An element $v$ of a $B$-comodule $(V,\rho)$ is $B$-invariant if $\rho(v)=v\otimes 1$. We denote the $k$-submodule of invariant elements by $V^B$. In other words, $V^B$ is the equalizer of the coaction $\rho$ and $id_V\otimes 1:V\to V\otimes B$. If $B$ is the coordinate ring of an affine group scheme $K$, we will denote $V^B$ by $H^0(K,V)$.
\end{defn}
\begin{prop}[\cite{MR2015057} I.2.10]\label{3.2.6}
Let $V$ be a $B$-comodule, and $W$ be a $k$-module. Then:
\begin{enumerate}
\renewcommand{\labelenumi}{(\arabic{enumi})}
\item $W$ is a $B$-comodule for $w\mapsto w\otimes 1$. This is called a trivial comodule.
\item There is a natural bijection $\Hom_B(W,V)\cong\Hom_k(W,V^B)$. 
\item We have a natural identification $V^B=\Hom_B(k,V)$.
\end{enumerate}
\end{prop}
\begin{proof}
Regard $k$ as a coalgebra over $k$. Then $W$ is a comodule over $k$ in the obvious way. Since the given map $k\to B$ is a homomorphism of coalgebras, it induces a coaction of $B$ on $W$ which coincides with (1).

Part (2) is obvious by definition: Every $B$-comodule homomorphism $f:W\to V$ is valued in $V^B$. Then (3) is obtained by applying $W=k$.
\end{proof}
\begin{var}\label{3.2.7}
Let $(\fg,K)$ be a pair over a commutative ring $k$, and $V$ be a $(\fg,K)$-module. Then $H^0(\fg,K,V)$ is naturally identified with the intersection of $H^0(K,V)$ and the $\fg$-invariant part of $V$.
\end{var}
\begin{lem}[\cite{MR2015057} I.2.10]\label{3.2.8}
Let $B$ be a bialgebra, $V$ be a $B$-comodule over a commutative ring $k$, and $k'$ be a flat $k$-algebra. Then we have
\[V^B\otimes k'\cong (V\otimes k')^{B\otimes k'}.\]
\end{lem}
\begin{proof}
Think of $V^B$ as $\Ker(\rho-id_V\otimes 1:V\to V\otimes B)$.
\end{proof}
\begin{cor}\label{3.2.9}
Let $K$ be a flat affine group scheme, and $V,V'$ be $K$-modules. Then there is a canonical isomorphism $H^0(K,F(V,V'))\cong\Hom_K(V,V')$.
\end{cor}
\begin{proof}
It follows from the natural identification
\[H^0(K,F(V,V'))=\Hom_K(k,F(V,V'))\cong\Hom_K(V,V').\]
\end{proof}
\begin{cor}\label{3.2.10}
Let $K$ be a flat affine group scheme, and $Q$ be a representation of $K$. Suppose that $Q$ is finitely presented as a $k$-module. Then $\Hom_K(Q,-)$ satisfies the flat base change formula: For any flat $k$-algebra $k'$, there is a canonical isomorphism
\[\Hom_K(Q,-)\otimes k'\cong\Hom_{K\otimes k'}(Q\otimes k',-\otimes k').\]
\end{cor}
\begin{proof}
It is immediate from Corollary \ref{3.2.9}, Proposition \ref{2.2.2} (2), and Lemma \ref{3.2.8}.
\end{proof}
\begin{var}\label{3.2.11}
Let $K$ be a flat affine group scheme over a commutative ring $k$, and $k'$ be a $k$-algebra which is finitely generated and projective as a $k$-module. Then we have a natural isomorphism
\[\Hom_K(-,-)\otimes k'\cong\Hom_{K\otimes k'}(-\otimes k',-\otimes k')\]
on $K\cmod^{op}\times K\cmod$.
\end{var}
\begin{proof}
Replacing $k'$ by a finitely generated and projective $k$-module $W$, we may prove
\[\Hom_K(-,-)\otimes W\cong\Hom_K(-,-\otimes W).\]
Here $W$ is regarded as a trivial $K$-module. It reduces to the cases where $W$ is free of finite rank by passing to retracts. Then the assertion follows since $\Hom_K(-,-)$ is additive in the second variable.
\end{proof}
\begin{proof}[Proof of Theorem \ref{3.1.6}]
Let $S$ be the collection of objects $X$ of $(\fg,K)\cmod$ such that $\Hom_{\fg,K}(X,-)$ satisfies the flat base change formula. Recall that $(\fg,K)\cmod$ is a compactly generated category whose compact objects are the finitely generated $(\fg,K)$-modules (Proposition \ref{3.2.4}). Since $S$ is closed under formation of finite colimits, it will suffice to show $\ind^\fg_\fk Q\in S$, where $Q$ is a $K$-module which is finitely generated as a $k$-module (Corollary \ref{2.1.12}, Lemma \ref{2.1.8}, Lemma \ref{2.1.11}). For any $(\fg,K)$-module $W$, we have
\[\begin{split}
\Hom_{\fg,K}(\ind^\fg_\fk Q,W)\otimes k'
&\cong \Hom_K(Q,W)\otimes k'\\
&\cong \Hom_{K\otimes k'}(Q\otimes k',W\otimes k')\\
&\cong\Hom_K(Q,\Res^k_{k'}(W\otimes k'))\\
&\cong \Hom_{\fg,K}(\ind^\fg_\fk Q,\Res^k_{k'}(W\otimes k'))\\
&\cong\Hom_{\fg\otimes k',K\otimes k'}(\ind^\fg_\fk Q\otimes k',W\otimes k')
\end{split}\]
(see Corollary \ref{3.2.10}). This completes the proof.
\end{proof}
\begin{proof}[Proof of Theorem \ref{3.1.7}]
On the other hand, according to Lemma \ref{2.1.8}, Proposition \ref{3.2.4}, and Definition \ref{2.1.7} (d), it will suffice to show that for any finitely generated $(\fg,K)$-module $V$ and a $(\fq,M)$-module $W$, the $k'$-homomorphism induced from $\iota$ in Construction \ref{3.1.5}
\[\Hom_{\fg\otimes k',K\otimes k'}(V\otimes k',I^{\fg,K}_{\fq,M}(W)\otimes k')
\to\Hom_{\fg\otimes k',K\otimes k'}(V\otimes k',I^{\fg\otimes k',K\otimes k'}_{\fq\otimes k',M\otimes k'}(W\otimes k'))\]
is a bijection.

On the other hand, notice that the assumption (i) implies that the forgetful functor $\cF^{\fq,M}_{\fg,K}$ respects compact objects. We therefore have a bijection
\[\begin{split}
\Hom_{\fg\otimes k',K\otimes k'}(V\otimes k',I^{\fg,K}_{\fq,M}(W)\otimes k')
&\cong\Hom_{\fg,K}(V,I^{\fg,K}_{\fq,M}(W))\otimes k'\\
&\cong\Hom_{\fq,M}(V,W)\otimes k'\\
&\cong\Hom_{\fq\otimes k',M\otimes k'}(V\otimes k',W\otimes k')\\
&\cong\Hom_{\fg\otimes k',K\otimes k'}(V\otimes k',I^{\fg\otimes k',K\otimes k'}_{\fq\otimes k',M\otimes k'}(W\otimes k')).
\end{split}\]
The assertion is reduced to showing that these two arrows coincide.

Observe that the adjunction of $(\cF^{\fq\otimes k',M\otimes k'}_{\fg\otimes k',K\otimes k'},I^{\fg\otimes k',K\otimes k'}_{\fq\otimes k',M\otimes k'})$ is $k'$-linear since so is $\cF^{\fq\otimes k',M\otimes k'}_{\fg\otimes k',K\otimes k'}$, and the adjunctions are described by units and counits which are $k'$-homomorphisms by definition. Therefore the bijection
\[\Hom_{\fq\otimes k',M\otimes k'}(V\otimes k',W\otimes k')\\
\cong\Hom_{\fg\otimes k',K\otimes k'}(V\otimes k',I^{\fg\otimes k',K\otimes k'}_{\fq\otimes k',M\otimes k'}(W\otimes k'))\]
is $k'$-linear. It implies that the sequence of bijections above is $k'$-linear. Hence we may restrict the maps along
\[\Hom_{\fg,K}(V,I^{\fg,K}_{\fq,M}(W))\to\Hom_{\fg\otimes k',K\otimes k'}(V\otimes k',I^{\fg,K}_{\fq,M}(W)\otimes k').\]
In this case, for $f\in\Hom_{\fg,K}(V,I^{\fg,K}_{\fq,M}(W))$, $f\otimes 1\in\Hom_{\fg\otimes k',K\otimes k'}(V\otimes k',I^{\fg,K}_{\fq,M}(W)\otimes k')$ goes to the element in
\[\Hom_{\fg\otimes k',K\otimes k'}(V\otimes k',I^{\fg\otimes k',K\otimes k'}_{\fq\otimes k',M\otimes k'}(W\otimes k'))\cong\Hom_{\fg,K}(V,I^{\fg,K}_{\fq,M}(\Res^k_{k'}(W\otimes k')))\]
described as
\[V\overset{f}{\to}I^{\fg,K}_{\fq,M}(W)\to I^{\fg,K}_{\fq,M}(\Res^k_{k'}(W\otimes k')).\]
This coincide with $\iota\circ (f\otimes 1)$ by definition of $\iota$. This completes the proof.
\end{proof}
Notice that for a finitely generated and projective $k$-module $W$, the functor $-\otimes W$ respects small limits of $k$-modules. Hence similar arguments work in the finite setting:
\begin{var}\label{3.2.12}
Let $(\fg,K)$ be a pair over a commutative ring $k$, and $k\to k'$ be a ring homomorphism. Assume that $k'$ is finitely generated and projective as a $k$-module. Then we have an isomorphism 
\[\Hom_{\fg,K}(-,-)\otimes k'\cong \Hom_{\fg\otimes k',K\otimes k'}(-\otimes k',-\otimes k')\]
on $(\fg,K)\cmod^{op}\times(\fg,K)\cmod$.
\end{var}
\begin{var}\label{3.2.13}
Let $(\fq,M)\to (\fg,K)$ be a map of pairs over a commutative ring $k$, and $k\to k'$ be a ring homomorphism. Assume that $k'$ is finitely generated and projective as a $k$-module. Then the map
\[\iota:(I^{\fg,K}_{\fq,M}-)\otimes_k k'\to I^{\fg\otimes_k k',K\otimes_k k'}_{\fq\otimes_k k',M\otimes_k k'}(-\otimes_k k')\]
is an isomorphism.
\end{var}
To prove the derived base change theorems, we need to deal with injective and acyclic objects. Recall that if we are given a Grothendieck abelian category $\cA$ and its family $C$ of generators, an object $X\in\cA$ is injective if and only if it has a right lifting property with respect to monomorphisms to members of $C$. In particular, if $\cA$ is locally Noetherian, the following conditions are equivalent:
\begin{enumerate}
\renewcommand{\labelenumi}{(\alph{enumi})}
\item $X$ is injective;
\item $X$ has a right lifting property with respect to monomorphisms to members of $C$;
\item $X$ has a right lifting property with respect to monomorphisms between Noetherian objects.
\end{enumerate}
\begin{lem}\label{3.2.14}
Let $(\fg,K)$ be a pair over a Noetherian ring $k$, and $k'$ be a flat $k$-algebra. Suppose that $\fg$ is finitely generated over $k$. If $I$ is an injective $(\fg,K)$-module, so is $\Res^k_{k'}( I\otimes k')$.
\end{lem}
\begin{proof}
This is an immediate consequence of Theorem \ref{3.1.6}. In fact, for every injective homomorphism $A\to B$ of $(\fg,K)$-modules,
we have
\[\begin{split}
\Hom_{\fg,K}(B,\Res^k_{k'} (I\otimes k'))
&\cong\Hom_{\fg,K}(B,I)\otimes k'\\
&\twoheadrightarrow\Hom_{\fg,K}(A,I)\otimes k'\\
&\cong \Hom_{\fg,K}(A,\Res^k_{k'} (I\otimes k')).
\end{split}\]
\end{proof}
\begin{proof}[Proof of Theorem \ref{3.1.9}]
For a finitely generated $(\fg,K)$-module $X$, and a complex $I$ concentrated in nonnegative degrees of injective $(\fg,K)$-modules, we have
\[\begin{split}
\bR\Hom_{\fg,K}(X,I)\otimes k'
&\simeq\Hom_{\fg,K}(X,I)\otimes k'\\
&\cong\Hom_{\fg,K}(X,I\otimes k')\\
&=\bR\Hom_{\fg,K}(X,I\otimes k').
\end{split}\]
The general case is deduced by passing to shifts and finite colimits in $\Coh(\fg,K)$. This completes the proof.
\end{proof}
\begin{proof}[Proof of Theorem \ref{3.1.10}]
Let $I^\bullet$ be a complex bounded below of injective $(\fg,K)$-modules. Since $\Res^k_{k'}$ is exact and conservative on $(\fg\otimes k',K\otimes k')\cmod$, so it is on $D(\fg\otimes k',K\otimes k')$. Hence each $I^n\otimes k'$ is $I^{\fg\otimes k',K\otimes k'}_{\fq\otimes k',M\otimes k'}$-acyclic (Corollary \ref{3.1.4} and Lemma \ref{3.2.14}). Corollary \ref{3.1.7} now implies
\[\begin{split}
(\bR I^{\fg,K}_{\fq,M}I^\bullet)\otimes_k k'
&= I^{\fg,K}_{\fq,M}(I^\bullet)\otimes k'\\
&\cong I^{\fg\otimes k',K\otimes k'}_{\fq\otimes k',M\otimes k'}(I^\bullet\otimes k')\\
&\simeq\bR I^{\fg\otimes k',K\otimes k'}_{\fq\otimes k',M\otimes k'}(I^\bullet\otimes_k k').
\end{split}\]
This completes the proof.
\end{proof}
Variant \ref{G} (3) is deduced from the following finite variant of Lemma \ref{3.2.14}:
\begin{lem}\label{3.2.15}
Let $(\fg,K)$ be a pair over a commutative ring, and $Q$ be a $(\fg,K)$-module which is finitely generated and projective as a $k$-module. Then $-\otimes Q$ respects injectively fibrant complexes of $(\fg,K)$-modules (see \cite{1606.04320}).
\end{lem}
\begin{proof}
We have a canonical isomorphism $Q\cong \Hom_k(\Hom_k(Q,k),k)$ of $(\fg,K)$-modules (see Proposition \ref{2.2.2}, Proposition \ref{2.2.3}). Hence we have a natural isomorphism
\[\begin{split}
\Hom_{\fg,K}(-,-\otimes Q)
&\cong\Hom_{\fg,K}(-,-\otimes \Hom_k(\Hom_k(Q,k),k))\\
&\cong\Hom_{\fg,K}(-,\Hom_k(\Hom_k(Q,k),-))\\
&\cong\Hom_{\fg,K}(-\otimes\Hom_k(Q,k),-).
\end{split}\]
The assertion now follows since $\Hom_k(Q,k)$ is flat as a $k$-module.
\end{proof}
\subsection{The unbounded derived version}
In this section, we replace $D(\fg,K)$ by another $\infty$-category to establish a generalization of Theorem \ref{3.1.10}. Regard $D(\fg,K)$ as the derived $\infty$-category, and set $\Ind\Coh(\fg,K)$ as the ind-completion of $\Coh(\fg,K)$ in the sense of \cite{MR2522659}.
Let $k\to k'$ be a flat ring homomorphism of Noetherian rings, and $(\fq,M)\to(\fg,K)$ be a map of pairs over $k$. Suppose that the following conditions are satisfied:
\begin{enumerate}
\renewcommand{\labelenumi}{(\roman{enumi})}
\item $\fk\oplus\fq\to\fg$ is surjective.
\item $\fq$ and $\fg$ are finitely generated as $k$-modules.
\end{enumerate}
\begin{lem}\label{3.3.1}
The functors
\[-\otimes k':D(\fg,K)\to D(\fg\otimes k',K\otimes k')\]
\[-\otimes k':D(\fq,M)\to D(\fq\otimes k',M\otimes k')\]
\[\cF^{\fq,M}_{\fg,K}:D(\fg,K)\to D(\fq,M)\]
respect coherent objects. In particular, they extend to left adjoint functors
\[-\otimes k':\Ind\Coh(\fg,K)\to \Ind\Coh(\fg\otimes k',K\otimes k')\]
\[-\otimes k':\Ind\Coh(\fq,M)\to \Ind\Coh(\fq\otimes k',M\otimes k')\]
\[\cF^{\fq,M}_{\fg,K}:\Ind\Coh(\fg,K)\to \Ind\Coh(\fq,M).\]
\end{lem}
\begin{proof}
It follows by definition. For $\cF^{\fq,M}_{\fg,K}$, use (i).
\end{proof}
Let us denote the resulting right adjoint functors as
\[\Res^{k,\ind}_{k'}:\Ind\Coh(\fg\otimes k',K\otimes k')\to\Ind\Coh(\fg,K)\]
\[\Res^{k,\ind}_{k'}:\Ind\Coh(\fq\otimes k',M\otimes k')\to\Ind\Coh(\fq,M)\]
\[I_{\fq,M}^{\fg,K,\ind}:\Ind\Coh(\fq,M)\to \Ind\Coh(\fg,K).\]
\begin{rem}[The second adjoint functor]\label{3.3.2}
Since $\cF^{\fq,M}_{\fg,K}$ is a proper left adjoint functor between compactly generated stable $\infty$-categories, $I_{\fq,M}^{\fg,K,\ind}$ admits a right adjoint functor (\cite{MR2522659} Corollary 5.5.2.9 (1)).
\end{rem}
To see the relation of our new right adjoint functors with the classical derived functors, recall that the standard $t$-structure on $D(\fg,K)$ descends to $\Coh(\fg,K)$, and then extends to $\Ind\Coh(\fg,K)$.
\begin{lem}\label{3.3.3}
\begin{enumerate}
\renewcommand{\labelenumi}{(\arabic{enumi})}
\item The functors
\[-\otimes k':\Ind\Coh(\fg,K)\to \Ind\Coh(\fg\otimes k',K\otimes k')\]
\[-\otimes k':\Ind\Coh(\fq,M)\to \Ind\Coh(\fq\otimes k',M\otimes k')\]
\[\cF^{\fq,M}_{\fg,K}:\Ind\Coh(\fg,K)\to \Ind\Coh(\fq,M)\]
are $t$-exact.
\item The functors
\[\Res^{k,\ind}_{k'}:\Ind\Coh(\fg\otimes k',K\otimes k')\to\Ind\Coh(\fg,K)\]
\[\Res^{k,\ind}_{k'}:\Ind\Coh(\fq\otimes k',M\otimes k')\to\Ind\Coh(\fq,M)\]
\[I_{\fq,M}^{\fg,K,\ind}:\Ind\Coh(\fq,M)\to \Ind\Coh(\fg,K)\]
are left $t$-exact.
\end{enumerate}
In particular, the adjunctions restrict to the eventually coconnective part.
\end{lem}
\begin{proof}
Part (1) follows since 
\[-\otimes k':\Coh(\fg,K)\to \Coh(\fg\otimes k',K\otimes k')\]
\[-\otimes k':\Coh(\fq,M)\to \Coh(\fq\otimes k',M\otimes k')\]
\[\cF^{\fq,M}_{\fg,K}:\Coh(\fg,K)\to \Coh(\fq,M)\]
are $t$-exact. Then (2) is immediate from the generalities on $t$-structures.
\end{proof}
Recall that for a stable $\infty$-category $\cC$ with a coherent $t$-structure, there is a canonical equivalence $\Ind\Coh(\cC)^+\simeq\cC^+$ (\cite{MR3730514} Proposition 6.3.2). If we restrict the diagram
\[\xymatrix{
\Ind\Coh(\fg,K)\ar[rr]^{-\otimes k'}\ar[d]&&\Ind\Coh(\fg\otimes k',K\otimes k')\ar[d]\\
D(\fg,K)\ar[rr]_{-\otimes k'}&&D(\fg\otimes k',K\otimes k')
}\]
to the eventually coconnective part, the vertical arrows are equivalences. Passing to the right adjoint, we conclude that $\Res^{k,\ind}_{k'}$ coincides with $\Res^k_{k'}$ on $D(\fg\otimes k',K\otimes k')^+$ and $D(\fq\otimes k',M\otimes k')^+$ under the identification. Similarly, we have $I^{\fg,K,\ind}_{\fq,M}|_{D(\fq,M)^+}\simeq \bR I^{\fg,K}_{\fq,M}|_{D(\fq,M)^+}$.
\begin{thm}\label{3.3.4}
The comparison map $\iota:I^{\fg,K,\ind}_{\fq,M}(-)\otimes k'\to I^{\fg\otimes k',K\otimes k',\ind}_{\fq\otimes k',M\otimes k'}(-\otimes k')$ is a natural isomorphism. Moreover, it restricts to the natural isomorphism $\bR I^{\fg,K}_{\fq,M}(-)\otimes k'\simeq \bR I^{\fg\otimes k',K\otimes k'}_{\fq\otimes k',M\otimes k'}(-\otimes k')$ of Theorem \ref{3.1.10} under the identifications
\[\Ind\Coh(\fq,M)^+\simeq D(\fq,M)^+\]
\[\Ind\Coh(\fg\otimes k',K\otimes k')^+\simeq D(\fg\otimes k',K\otimes k')^+\]
\end{thm}
\begin{proof}
Since the functors are continuous, we may prove the natural isomorphism on $\Coh(\fq,M)$. Then the assertion is reduced to Theorem \ref{3.1.10} since Construction \ref{3.1.5} is compatible with our ind-setting under the equivalences of the type $\Ind\Coh(\cC)^+\simeq\cC^+$ (recall the compatibility of the adjunctions of $-\otimes k'$ and $\cF$ in the two settings from Lemma \ref{3.3.1} and the argument below there).
\end{proof}
Finally, suppose that $k$ is a field of characteristic 0, $(\fg,K)$ be a pair with $K$ reductive and $\dim \fg<+\infty$.
\begin{prop}\label{3.3.5}
The embedding $\Coh(\fg,K)\to D(\fg,K)$ induces an equivalence $\Ind\Coh(\fg,K)\simeq D(\fg,K)$.
\end{prop}
\begin{proof}
If we are given an arbitrary finitely generated $(\fg,K)$-module $V$, there is a finite dimensional $K$-submodule $V_0$ such that the induced homomorphism $U(\fg)\otimes_{U(\fk)}V_0\to V$ is surjective. Since its kernel is also finitely generated, we can repeat this procedure to obtain a resolution of $V$ by finitely generated and projective $(\fg,K)$-modules. According to the existence of the standard projective resolution (\cite{1604.04253} 1.4.4), the category $(\fg,K)\cmod$ has a finite homological dimension. In particular, we may assume the resolution to be bounded by truncations. Moreover, it implies that $V$ is compact in the $\infty$-category $D(\fg,K)$. Passing to shifts and finite colimits, we can conclude that every coherent complex is compact in $D(\fg,K)$. Since $\Coh(\fg,K)$ generates $D(\fg,K)$ under colimits, the equivalence follows (\cite{MR2522659} Proposition 5.3.5.11, Proposition 5.5.1.9).
\end{proof}
\section{Variants for $\pro$}
\subsection{Computation of $\pro$}
\begin{lem}\label{4.1.1}
Let $K$ be a flat affine group scheme over $k$, and $\{V_\cO\}_\cO$ be a set of $K$-modules. Suppose that for any finitely generated $K$-module $Q$, $\Hom_K(Q,V_\cO)$ vanishes for all but finitely many indices $\cO$. Then the direct sum $\oplus V_\cO$ also exhibits a product of $\{V_\cO\}$ in $K\cmod$.
\end{lem}
\begin{proof}
It is obvious since we have a bijection for any finitely generated $K$-module $Q$
\[\Hom_K(Q,\oplus V_\cO)
\cong\oplus\Hom_K(Q,V_\cO)\cong\prod\Hom_K(Q,V_\cO).\]
The second one follows from the assumption on $\{V_\cO\}$. Since such $Q$ form a family of generators, the assertion follows (Corollary \ref{2.1.12}, Definition \ref{2.1.7}).
\end{proof}
We give a characterization of the assumption above in practical settings.
\begin{prop}\label{4.1.2}
Let $k$ be a Noetherian domain, $K$ be a flat affine group scheme over a Noetherian ring $k$, and $V=\oplus V_\cO$ be a direct sum of $K$-modules.
\begin{enumerate}
\renewcommand{\labelenumi}{(\arabic{enumi})}
\item If $V$ is torsion-free, and $V\otimes\Frac(k)$ is admissible then for any finitely generated $K$-module $Q$,
$\Hom_K(Q, V_\cO)$ vanishes for all but finitely many indices $\cO$.
\item If for any finitely generated $K$-module $Q$, $\Hom_K(Q,V)$ is finitely generated then $V\otimes\Frac(k)$ is admissible.
\item Suppose that each of $V_\cO$ is finitely generated. If for a finitely generated $K$-module $Q$, $\Hom_K(Q,V_\cO)$ vanishes for all but finitely many indices $\cO$ then $\Hom_K(Q,V)$ is finitely generated.
\end{enumerate}
\end{prop}
\begin{proof}
To see (1), consider a sequence for a finitely generated $K$-module $Q$
\[\begin{split}
\Hom_K(Q,V)
&=\oplus\Hom_K(Q,V_\cO)\\
&\subset\oplus\Hom_{K\otimes\Frac(k)}(Q\otimes\Frac(k),V_\cO\otimes\Frac(k))\\
&\cong\Hom_{K\otimes\Frac(k)}(Q\otimes\Frac(k),V\otimes\Frac(k)).
\end{split}\]
Since $V\otimes\Frac(k)$ is admissible, $\Hom_K(Q,V_\cO\otimes\Frac(k))$ vanishes for all but finitely many $\cO$. Since $V_\cO$ are torsion-free, the submodules $\Hom_K(Q,V_\cO)$ vanish for almost all $\cO$.

Part (2) follows from the flat base change theorem: For any finitely generated $K$-module $Q$, we have
\[\dim\Hom_{K\otimes\Frac(k)}(Q\otimes\Frac(k),V\otimes\Frac(k))
=\dim\Hom_K(Q,V)\otimes\Frac(k)<+\infty.\]
Since finite dimensional representations of $K\otimes\Frac(k)$-modules are generated by representations $Q\otimes\Frac(k)$ under finite colimits (Corollary \ref{2.1.12}), $V\otimes\Frac(k)$ is admissible.

Finally, suppose that $V_\cO$ are finitely generated. Then for a finitely generated $K$-module $Q$, $\Hom_K(Q,V)$ is isomorphic to a direct sum of $\Hom_K(Q,V_\cO)$ along finitely many indices $\cO$. Since $V_\cO$ is finitely generated, so is $\Hom_K(Q,V)$. This completes the proof.
\end{proof}
\begin{prop}\label{4.1.3}
Let $(\fq,M)\to(\fg,M)$ be an injective map of pairs over a Noetherian ring $k$, and $Z$ be a $(\fq,M)$-module. Suppose that the map $M\to M$ is the identity. Moreover, assume the following conditions:
\begin{enumerate}
\renewcommand{\labelenumi}{(\roman{enumi})}
\item For $x\in\fg$, we have $\left[x,x\right]=0$.
\item There is an $M$-equivariant Lie subalgebra $\bar{\fu}\subset\fg$ such that the summation map $\fq\oplus\bar{\fu}\to\fg$ is an isomorphism of $k$-modules.
\item There are free bases of $\fq$ and $\bar{\fu}$.
\item The enveloping algebra $U(\bar{\fu})$ is decomposed into a direct sum $U(\bar{\fu})=\oplus_\cO U(\bar{\fu})_\cO$ of $M$-submodules $U(\bar{\fu})_\cO$ which are finitely generated as $k$-modules.
\item For any finitely generated $M$-module $Q$, $\Hom_{M}(Q,\Hom(U(\bar{\fu})_\cO,Z))$ vanishes for all but finitely many $\cO$.
\end{enumerate}
Then we have an isomorphism as an $M$-module
\[\pro^\fg_{\fq}(Z)\cong\oplus_\cO\Hom_k(U(\bar{\fu})_\cO,Z).\]
In particular, a base change formula along a ring homomorphism $k\to k'$ between Noetherian rings
\[\pro^\fg_\fq(Z)\otimes k'\cong\pro^{\fg\otimes k'}_{\fq\otimes k'}(Z\otimes k')\]
is valid in the following cases:
\begin{enumerate}
\renewcommand{\labelenumi}{(\alph{enumi})}
\item $k\to k'$ is flat.
\item For any finitely generated $M\otimes k'$-module $Q$, $\Hom_{M\otimes k'}(Q,\Hom(U(\bar{\fu})_\cO\otimes k',Z\otimes k'))$ vanishes for all but finitely many $\cO$.
\end{enumerate}
\end{prop}
\begin{rem}\label{4.1.4}
The functor $\pro^\fg_\fq$ can be regarded as a right adjoint functor to the forgetful functor from the category of weak $(\fg,M)$-modules to that of weak $(\fq,M)$-modules. Therefore the base change functor along arbitrary ring homomorphisms makes sense (Remark \ref{3.1.3}).
\end{rem}
\begin{proof}[Proof of Proposition \ref{4.1.3}]
According to the PBW theorem (\cite{MR979493}), we have an isomorphism of $M$-modules
\[\pro^\fg_\fq(Z)\cong F(U(\bar{\fu}),Z).\]
The condition (v) and Lemma \ref{4.1.1} imply that $F(U(\bar{\fu}),Z)\cong\oplus_\cO\Hom(U(\bar{\fu})_\cO,Z)$. This completes the proof.
\end{proof}
\subsection{Examples}
Suppose that we are given a reductive pair $(\fg_\bC,K_\bC)$ over $\bC$ and a $\theta$-stable parabolic subpair $(\fq_\bC,(K_L)_\bC)$ in the sense of \cite{MR1330919}, where $\theta$ is the Cartan involution. Let $\fl_\bC$ (resp.\ $\fu_\bC$, $\bar{\fu}_\bC$) denote the Levi part (resp.\ nilradical, the opposite nilradical) of $\fq$, let $\Delta(\bar{\fu}_\bC)=\{\alpha_1,\cdots,\alpha_s\}$ be the set of roots in $\bar{\fu}_\bC$, and $h=h_{\rho(\fu_\bC)}$ be the element of the Cartan subalgebra as in \cite{MR1330919} Proposition 4.70. In particular, we have $\alpha_i(h)<0$ for $\alpha_i\in\Delta(\bar{\fu}_\bC)$.
\begin{ex}\label{4.2.1}
Observe that $(\fl_\bC,(K_L)_\bC)\subset(\fq_\bC,(K_L)_\bC)$ are $\theta$-stable subpairs of $(\fg_\bC,K_\bC)$, where $\fl_\bC$ is the Levi part of $\fq_\bC$. Note that $\fu_\bC$ is also $\theta$-stable. Therefore they associate maps of contraction families over the polynomial ring $\bC\left[z\right]$
\[(\tilde{\fl}_\bC,(K_L)_\bC\otimes\bC\left[z\right])\gets(\tilde{\fq}_\bC,(K_L)_\bC\otimes\bC\left[z\right])\to(\tilde{\fg}_\bC,K_\bC\otimes\bC\left[z\right])\]
in the sense of \cite{10.1093/imrn/rny146} and \cite{10.1093/imrn/rny147}. Define the cohomological induction as
\[\bR I^{\tilde{\fg}_\bC,K_\bC\otimes\bC\left[z\right]}_{\tilde{\fq}_\bC,(K_L)_\bC\otimes\bC\left[z\right]}\cF^{\tilde{\fq}_\bC,(K_L)_\bC\otimes\bC\left[z\right]}
_{\tilde{\fl}_\bC,(K_L)_\bC\otimes\bC\left[z\right]}(-\otimes_{\bC\left[z\right]}\wedge^{\dim\fu}\tilde{\fu}),\]
where $\cF^{\tilde{\fq}_\bC,(K_L)_\bC\otimes\bC\left[z\right]}
_{\tilde{\fl}_\bC,(K_L)_\bC\otimes\bC\left[z\right]}$ is the forgetful functor
\[(\tilde{\fl}_\bC,(K_L)_\bC\otimes\bC\left[z\right])\cmod\to(\tilde{\fq}_\bC,(K_L)_\bC\otimes\bC\left[z\right])\cmod.\]
Remark that $\pro^{\tilde{\fg}_\bC}_{\tilde{\fq}_\bC}$ is exact (\cite{1712.07336} Variant 2.6, Corollary 2.12). Let $Z$ be a torsion-free $(\tilde{\fl}_\bC,(K_L)_\bC\otimes\bC\left[z\right])$-module with a scalar action of $h$. If $Z\otimes\bC(z)$ is admissible, the cohomological induction enjoys a flat base change formula to the algebraic closure $\overline{\bC(z)}$ of the field of rational functions $\bC(z)$
\begin{flalign*}
&\bR I^{\tilde{\fg}_\bC,K_\bC\otimes\bC\left[z\right]}_{\tilde{\fq}_\bC,(K_L)_\bC\otimes\bC\left[z\right]}\cF^{\tilde{\fq}_\bC,(K_L)_\bC\otimes\bC\left[z\right]}
_{\tilde{\fl}_\bC,(K_L)_\bC\otimes\bC\left[z\right]}(-\otimes_{\bC\left[z\right]}\wedge^{\dim\fu}\tilde{\fu})\otimes\overline{\bC(z)}\\
&\cong\bR I^{\fg\otimes\overline{\bC(z)},K_\bC\otimes\overline{\bC(z)}}_{\fq_\bC\otimes\overline{\bC(z)},
(K_L)_\bC\otimes\overline{\bC(z)}}\cF^{\fq_\bC\otimes\overline{\bC(z)},
(K_L)_\bC\otimes\overline{\bC(z)}}
_{\fl_\bC\otimes\overline{\bC(z)},(K_L)_\bC\otimes\overline{\bC(z)}}
Z\otimes\wedge^{\dim\fu}\fu\otimes\overline{\bC(z)})
\end{flalign*}
(use \cite{MR1330919} Proposition 5.96). Suppose also that the $\tau$-type $Z_\tau\subset Z$ for each irreducible representation $\tau$ of $(K_L)_\bC$ is free of finite rank over $\bC\left[z\right]$. Then for any $\bC$-algebra homomorphism $\bC\left[z\right]\to\bC$, we have a base change formula
\[\pro^{\tilde{\fg}_\bC}_{\tilde{\fq}_\bC}(Z)\otimes_{\bC\left[z\right]}\bC
\cong
\pro^{\tilde{\fg}_\bC\otimes_{\bC\left[z\right]}\bC}_{\tilde{\fq}_\bC\otimes_{\bC\left[z\right]}\bC}
(Z\otimes_{\bC\left[z\right]}\bC).\]
\end{ex}
Let $k$ be a Noetherian subring of $\bC$, and $(\fq,K_L)\subset(\fg,K)$ be a $k$-form of $(\fq_\bC,(K_L)_\bC)\subset(\fg_\bC,K_\bC)$ in the sense that $(\fq_\bC,(K_L)_\bC)\subset(\fg_\bC,K_\bC)$ is isomorphic to the base change of $(\fq,K_L)\subset(\fg,K)$. Assume that there is a complementary $K_L$-stable subalgebra $\bar{\fu}\subset\fg$ to $\fq$ which is a $k$-form of $\bar{\fu}_\bC$. Moreover, suppose that the following conditions are satisfied:
\begin{enumerate}
\renewcommand{\labelenumi}{(\roman{enumi})}
\item There is a free basis of $\fq$.
\item There is a free basis $\{E_{\alpha_i}\}$ of $\bar{\fu}$ consisting of root vectors of $\bar{\fu}_\bC$.
\item The $(K_L)_\bC$-orbit of $h$ is contained in the Cartan subalgebra.
\end{enumerate}
\begin{prop}\label{4.2.2}
In this setting, there is a family $\{U(\bar{\fu})_\cO\}$ of finitely generated $K_L$-submodules of $U(\bar{\fu})$ such that
\[U(\bar{\fu})=\oplus_\cO U(\bar{\fu})_\cO.\]
\end{prop}
\begin{cons}\label{4.2.3}
Let $G$ be the component group $\pi_0((K_L)_\bC)$ of $(K_L)_\bC$. For each $x\in G$, fix a representative $g_x\in (K_L)_\bC$ and set $h_x=\Ad(g_x)h$, where $\Ad$ is the action of $(K_L)_\bC$ on $\fg_\bC$. Since the unit component $(K_L)^0_\bC$ centralizes $h$, it is independent of the choice of $g_x$. In particular, if $g=e$ is the unit then $h_e=h$. Observe next that $G$ acts on the complex vector space $\bC^G$ by translation of entries. For a $G$-orbit $\cO$ in $\bC^G$, define $U(\bar{\fu})_\cO$ as
\[U(\bar{\fu})_\cO=\oplus_{\vec{r}\in\cO}\oplus_{\stackrel{\sum n_i\alpha_i(h_x)=r_x}{{\rm for\ any\ }x\in G}}kE^{n_1}_{\alpha_1}E^{n_2}_{\alpha_2}\cdots E^{n_s}_{\alpha_s}.\]
We also set
\[U(\bar{\fu})_{\vec{r}}=\oplus_{\stackrel{\sum n_i\alpha_i(h_x)=r_x}{{\rm for\ any\ }x\in G}}kE^{n_1}_{\alpha_1}E^{n_2}_{\alpha_2}\cdots E^{n_s}_{\alpha_s}.\]
\end{cons}
\begin{proof}[Proof of Proposition \ref{4.2.2}]
The $k$-modules $U(\bar{\fu})_\cO$ are finitely generated by definition. According to the PBW theorem, we have $U(\bar{\fu})\cong \oplus_\cO U(\bar{\fu})_\cO$ as a $k$-module. To see that $U(\bar{\fu})_\cO$ is a $K_L$-submodule, we may assume $k=\bC$ (Lemma \ref{2.1.14}). 

Fix $\vec{r}\in\cO$ and exponents $\{n_i\}$ with $\sum n_i\alpha_i(h_x)=r_x$. Let $g\in(K_L)_\bC$ in a component $x\in G$, and write
\[\Ad(g)E^{n_1}_{\alpha_1}E^{n_2}_{\alpha_2}\cdots E^{n_s}_{\alpha_s}
=\sum_{\vec{r}'\in\bC^G}v'_{\vec{r}'}\]
with $v'_{\vec{r}'}\in U(\bar{\fu})_{\vec{r}'}$. Then for any $y\in G$,
\[\begin{split}
\sum_{\vec{r}'\in\bC^G} r'_y v_{\vec{r}'}
&=\sum\left[h_y,v_{\vec{r}'}\right]\\
&=\left[h_y,\Ad(g)E^{n_1}_{\alpha_1}E^{n_2}_{\alpha_2}\cdots E^{n_s}_{\alpha_s}\right]\\
&=\Ad(g)\left[\Ad(g)^{-1}h_y,E^{n_1}_{\alpha_1}E^{n_2}_{\alpha_2}\cdots E^{n_s}_{\alpha_s}\right]\\
&=\Ad(g)\left[h_{x^{-1}y},E^{n_1}_{\alpha_1}E^{n_2}_{\alpha_2}\cdots E^{n_s}_{\alpha_s}\right]\\
&=r_{x^{-1}y}\Ad(g)E^{n_1}_{\alpha_1}E^{n_2}_{\alpha_2}\cdots E^{n_s}_{\alpha_s}\\
&=r_{x^{-1}y}\sum_{\vec{r}'\in\bC^G}v_{\vec{r}'}.
\end{split}\]
Therefore $v_{\vec{r'}}$ vanishes unless $\vec{r}'=x^{-1}\cdot \vec{r}$. In particular, $\Ad(g)E^{n_1}_{\alpha_1}E^{n_2}_{\alpha_2}\cdots E^{n_s}_{\alpha_s}\in U(\bar{\fu})_{x^{-1}\vec{r}}\subset U(\bar{\fu})_\cO$. This completes the proof.
\end{proof}
\begin{ex}\label{4.2.4}
Let $0\leq p\leq q$ be nonnegative integers with $n=p+q\geq 1$. Then the diagonal embedding $\GL_p\times\GL_q\to\GL_{p+q}$ gives rise to a pair $(\fgl_{p+q},\GL_p\times\GL_q)$ over $\bZ$. Consider a partition $n=\sum_{i=1}^l p_i$ with $p_i\geq 1$ for $1\leq i\leq l$. Write $m_0=\min\{m\geq 1:p<\sum_{i=1}^m p_i\}$. Set $Q$ as the subgroup of $\GL_{p+q}$ consisting of upper triangular block matrices of size $p_1 \times p_2 \times \cdots \times p_l$. Let $K_L$ be the subgroup of $Q$ consisting of matrices whose entries are zero outside
\begin{flalign*}
&\bigcup_{m=1}^{m_0-1}\{(a,b)\in\{1,2\cdots,n\}^2:1+\sum_{i=1}^{m-1}p_i\leq a,b\leq\sum_{i=1}^m p_i\}\\
&\cup\{(a,b)\in\{1,2\cdots,n\}^2:1+\sum_{i=1}^{m_0-1}p_i\leq a,b\leq p\}\\
&\cup\{(a,b)\in\{1,2\cdots,n\}^2:p+1\leq a,b\leq\sum_{i=1}^{m_0} p_i\}\\
&\cup\bigcup_{m=m_0+1}^l\{(a,b)\in\{1,2\cdots,n\}^2:1+\sum_{i=1}^{m-1}p_i\leq a,b\leq\sum_{i=1}^m p_i\}.
\end{flalign*}
Then $\fq$ and $K_L$ form a subpair of $(\fgl_{p+q},\GL_p\times\GL_q)$. Moreover, it is an integral model of the pair associated to $U(p,q)$ and a $\theta$-stable parabolic subpair. Moreover, it enjoys the conditions above.
\end{ex}
For (v) in Proposition \ref{4.1.3}, let $Z$ be a torsion-free $(\fq,K_L)$-module. See also Proposition \ref{4.1.2} (1).
\begin{prop}[\cite{MR1330919} Proposition 5.96]\label{4.2.5}
If $Z\otimes\bC$ is admissible, and that $h$ acts on $Z\otimes\bC$ as a scalar then the $(K_L)_\bC$-module $\oplus\Hom(U(\bar{\fu})_\cO, Z\otimes\bC)$ is admissible.
\end{prop}


\begin{thebibliography}{10}

\bibitem{MR1294136}
J.~Ad\'{a}mek and J.~Rosick\'{y}.
\newblock {\em Locally presentable and accessible categories}, volume 189 of
  {\em London Mathematical Society Lecture Note Series}.
\newblock Cambridge University Press, Cambridge, 1994.

\bibitem{MR3730514}
D.~Ben-Zvi, D.~Nadler, and A.~Preygel.
\newblock Integral transforms for coherent sheaves.
\newblock {\em J. Eur. Math. Soc. (JEMS)}, 19(12):3763--3812, 2017.

\bibitem{10.1093/imrn/rny146}
J.~Bernstein, N.~Higson, and E.~Subag.
\newblock {Contractions of Representations and Algebraic Families of
  Harish-Chandra Modules}.
\newblock {\em Int. Math. Res. Not. IMRN}, 07 2018.

\bibitem{10.1093/imrn/rny147}
J.~Bernstein, N.~Higson, and E.~Subag.
\newblock {Algebraic Families of Harish-Chandra Pairs}.
\newblock {\em Int. Math. Res. Not. IMRN}, 07 2018.

\bibitem{MR1291599}
F.~Borceux.
\newblock {\em Handbook of categorical algebra. 1}, volume~50 of {\em
  Encyclopedia of Mathematics and its Applications}.
\newblock Cambridge University Press, Cambridge, 1994.
\newblock Basic category theory.

\bibitem{MR1313497}
F.~Borceux.
\newblock {\em Handbook of categorical algebra. 2}, volume~51 of {\em
  Encyclopedia of Mathematics and its Applications}.
\newblock Cambridge University Press, Cambridge, 1994.
\newblock Categories and structures.

\bibitem{MR979493}
N.~Bourbaki.
\newblock {\em Lie groups and {L}ie algebras. {C}hapters 1--3}.
\newblock Elements of Mathematics (Berlin). Springer-Verlag, Berlin, 1989.
\newblock Translated from the French, Reprint of the 1975 edition.

\bibitem{MR2138086}
L.~El~Kaoutit, J.~G\'{o}mez-Torrecillas, and F.~J. Lobillo.
\newblock Semisimple corings.
\newblock {\em Algebra Colloq.}, 11(4):427--442, 2004.

\bibitem{1407.0574}
G.~Harder.
\newblock {H}arish-{C}handra {M}odules over {$\mathbb Z$}, 2014.
\newblock arXiv:1405.6513.

\bibitem{MR3053412}
M.~Harris.
\newblock Beilinson-{B}ernstein localization over {$\mathbb Q$} and periods of
  automorphic forms.
\newblock {\em Int. Math. Res. Not. IMRN}, (9):2000--2053, 2013.

\bibitem{10.1093/imrn/rny043}
M.~Harris.
\newblock Beilinson-{B}ernstein localization over {$\mathbb Q$} and periods of
  automorphic forms: Erratum.
\newblock {\em Int. Math. Res. Not. IMRN}, 03 2018.

\bibitem{1606.04320}
T.~Hayashi.
\newblock Dg analogues of the {Z}uckerman functors and the dual {Z}uckerman
  functors {II}, 2016.
\newblock arXiv:1606.04320.

\bibitem{1712.07336}
T.~Hayashi.
\newblock Integral models of {H}arish-{C}handra modules of the finite covering
  groups of {PU}(1,1), 2017.
\newblock arXiv:1712.07336.

\bibitem{MR4007195}
T.~Hayashi.
\newblock Dg analogues of the {Z}uckerman functors and the dual {Z}uckerman
  functors {I}.
\newblock {\em J. Algebra}, 540:274--305, 2019.

\bibitem{1001.1556}
K.~Hess.
\newblock A general framework for homotopic descent and codescent, 2010.
\newblock arXiv:1001.1556.

\bibitem{MR2066503}
M.~Hovey.
\newblock Homotopy theory of comodules over a {H}opf algebroid.
\newblock In {\em Homotopy theory: relations with algebraic geometry, group
  cohomology, and algebraic {$K$}-theory}, volume 346 of {\em Contemp. Math.},
  pages 261--304. Amer. Math. Soc., Providence, RI, 2004.

\bibitem{MR499562}
James~E. Humphreys.
\newblock {\em Introduction to {L}ie algebras and representation theory},
  volume~9 of {\em Graduate Texts in Mathematics}.
\newblock Springer-Verlag, New York-Berlin, 1978.
\newblock Second printing, revised.

\bibitem{MR2015057}
J.~C. Jantzen.
\newblock {\em Representations of algebraic groups}, volume 107 of {\em
  Mathematical Surveys and Monographs}.
\newblock American Mathematical Society, Providence, RI, second edition, 2003.

\bibitem{1604.04253}
F.~Januszewski.
\newblock On {P}eriod {R}elations for {A}utomorphic {$L$}-functions {II}, 2016.
\newblock arXiv:1604.04253.

\bibitem{MR3770183}
F.~Januszewski.
\newblock Rational structures on automorphic representations.
\newblock {\em Math. Ann.}, 370(3-4):1805--1881, 2018.

\bibitem{MR1330919}
A.~W. Knapp and D.~A. Vogan, Jr.
\newblock {\em Cohomological induction and unitary representations}, volume~45
  of {\em Princeton Mathematical Series}.
\newblock Princeton University Press, Princeton, NJ, 1995.

\bibitem{MR2157133}
H.~Krause.
\newblock The stable derived category of a {N}oetherian scheme.
\newblock {\em Compos. Math.}, 141(5):1128--1162, 2005.

\bibitem{MR2522659}
J.~Lurie.
\newblock {\em Higher topos theory}, volume 170 of {\em Annals of Mathematics
  Studies}.
\newblock Princeton University Press, Princeton, NJ, 2009.

\bibitem{MR0277536}
J.~Tits.
\newblock Repr\'{e}sentations lin\'{e}aires irr\'{e}ductibles d'un groupe
  r\'{e}ductif sur un corps quelconque.
\newblock {\em J. Reine Angew. Math.}, 247:196--220, 1971.

\bibitem{MR547117}
W.~C. Waterhouse.
\newblock {\em Introduction to affine group schemes}, volume~66 of {\em
  Graduate Texts in Mathematics}.
\newblock Springer-Verlag, New York-Berlin, 1979.

\end{thebibliography}
\end{document}